\documentclass[11pt]{amsart}
\usepackage{amsmath,amsfonts,mathrsfs}

\newif\ifdraft
\draftfalse
\usepackage{amsmath,amsfonts,amsthm,mathrsfs}
\usepackage{amssymb}

\usepackage[unicode,bookmarks]{hyperref}
\usepackage[usenames,dvipsnames]{xcolor}
\hypersetup{colorlinks=true,citecolor=NavyBlue,linkcolor=Brown,urlcolor=Orange}

\usepackage[alphabetic,initials]{amsrefs}

\usepackage{enumitem}

\usepackage{chngcntr}

\ifdraft
\usepackage[notcite,notref,color]{showkeys}

\definecolor{labelkey}{gray}{0.5}
\fi

\usepackage{tikz}

\usepackage{tikz-cd}
\usetikzlibrary{matrix,arrows}
\newlength{\myarrowsize} 

\pgfarrowsdeclare{cmto}{cmto}{
	\pgfsetdash{}{0pt} 
	\pgfsetbeveljoin 
	\pgfsetroundcap 
	\setlength{\myarrowsize}{0.6pt}
	\addtolength{\myarrowsize}{.5\pgflinewidth}
	\pgfarrowsleftextend{-4\myarrowsize-.5\pgflinewidth} 
	\pgfarrowsrightextend{.8\pgflinewidth}
}{
	\setlength{\myarrowsize}{0.6pt} 
  	\addtolength{\myarrowsize}{.5\pgflinewidth}  
	\pgfsetlinewidth{0.5\pgflinewidth}
	\pgfsetroundjoin
	\pgfpathmoveto{\pgfpoint{1.5\pgflinewidth}{0}}
	\pgfpatharc{-109}{-170}{4\myarrowsize}
	\pgfpatharc{10}{189}{0.58\pgflinewidth and 0.2\pgflinewidth}
	\pgfpatharc{-170}{-115}{4\myarrowsize+\pgflinewidth}
	\pgfpathclose
	\pgfusepathqfillstroke
	\pgfpathmoveto{\pgfpoint{1.5\pgflinewidth}{0}}
	\pgfpatharc{109}{170}{4\myarrowsize}
	\pgfpatharc{-10}{-189}{0.58\pgflinewidth and 0.2\pgflinewidth}
	\pgfpatharc{170}{115}{4\myarrowsize+\pgflinewidth}
	\pgfpathclose
	\pgfusepathqfillstroke
	\pgfsetlinewidth{2\pgflinewidth}
}

\pgfarrowsdeclare{cmonto}{cmonto}{
	\pgfsetdash{}{0pt} 
	\pgfsetbeveljoin 
	\pgfsetroundcap 
	\setlength{\myarrowsize}{0.6pt}
	\addtolength{\myarrowsize}{.5\pgflinewidth}
	\pgfarrowsleftextend{-4\myarrowsize-.5\pgflinewidth} 
	\pgfarrowsrightextend{.8\pgflinewidth}
}{
	\setlength{\myarrowsize}{0.6pt} 
  	\addtolength{\myarrowsize}{.5\pgflinewidth}  
	\pgfsetlinewidth{0.5\pgflinewidth}
	\pgfsetroundjoin
	\pgfpathmoveto{\pgfpoint{1.5\pgflinewidth}{0}}
	\pgfpatharc{-109}{-170}{4\myarrowsize}
	\pgfpatharc{10}{189}{0.58\pgflinewidth and 0.2\pgflinewidth}
	\pgfpatharc{-170}{-115}{4\myarrowsize+\pgflinewidth}
	\pgfpathclose
	\pgfusepathqfillstroke
	\pgfpathmoveto{\pgfpoint{1.5\pgflinewidth}{0}}
	\pgfpatharc{109}{170}{4\myarrowsize}
	\pgfpatharc{-10}{-189}{0.58\pgflinewidth and 0.2\pgflinewidth}
	\pgfpatharc{170}{115}{4\myarrowsize+\pgflinewidth}
	\pgfpathclose
	\pgfusepathqfillstroke
	\pgfpathmoveto{\pgfpoint{1.5\pgflinewidth-0.3em}{0}}
	\pgfpatharc{-109}{-170}{4\myarrowsize}
	\pgfpatharc{10}{189}{0.58\pgflinewidth and 0.2\pgflinewidth}
	\pgfpatharc{-170}{-115}{4\myarrowsize+\pgflinewidth}
	\pgfpathclose
	\pgfusepathqfillstroke
	\pgfpathmoveto{\pgfpoint{1.5\pgflinewidth-0.3em}{0}}
	\pgfpatharc{109}{170}{4\myarrowsize}
	\pgfpatharc{-10}{-189}{0.58\pgflinewidth and 0.2\pgflinewidth}
	\pgfpatharc{170}{115}{4\myarrowsize+\pgflinewidth}
	\pgfpathclose
	\pgfusepathqfillstroke
	\pgfsetlinewidth{2\pgflinewidth}
}

\pgfarrowsdeclare{cmhook}{cmhook}{
	\pgfsetdash{}{0pt} 
	\pgfsetbeveljoin 
	\pgfsetroundcap 
	\setlength{\myarrowsize}{0.6pt}
	\addtolength{\myarrowsize}{.5\pgflinewidth}
	\pgfarrowsleftextend{-4\myarrowsize-.5\pgflinewidth} 
	\pgfarrowsrightextend{.8\pgflinewidth}
}{
	\setlength{\myarrowsize}{0.6pt} 
  	\addtolength{\myarrowsize}{.5\pgflinewidth}  
 	\pgfsetdash{}{0pt}
	\pgfsetroundcap
	\pgfpathmoveto{\pgfqpoint{0pt}{-4.667\pgflinewidth}}
	\pgfpathcurveto
    {\pgfqpoint{4\pgflinewidth}{-4.667\pgflinewidth}}
    {\pgfqpoint{4\pgflinewidth}{0pt}}
    {\pgfpointorigin}
	\pgfusepathqstroke
}

\newenvironment{diagram*}[2]{%
\[%
\begin{tikzpicture}[>=cmto,baseline=(current bounding box.center),%
	to/.style={->,font=\scriptsize,cap=round},%
	into/.style={cmhook->,font=\scriptsize,cap=round},%
	onto/.style={-cmonto,font=\scriptsize,cap=round},%
	math/.style={matrix of math nodes, row sep=#2, column sep=#1,%
		text height=1.5ex, text depth=0.25ex}]%
}{%
\end{tikzpicture}%
\]% 
\ignorespacesafterend%
}

\newcommand{\Dmod}{\mathscr{D}}
\newcommand{\Mmod}{\mathcal{M}}
\newcommand{\Pmod}{\mathcal{P}}
\newcommand{\Nmod}{\mathcal{N}}

\newcommand{\derR}{\mathbf{R}}
\newcommand{\derL}{\mathbf{L}}

\newcommand{\ZZ}{\mathbb{Z}}
\newcommand{\QQ}{\mathbb{Q}}
\newcommand{\RR}{\mathbb{R}}
\newcommand{\CC}{\mathbb{C}}

\newcommand{\PP}{\mathbb{P}}

\DeclareMathOperator{\Spec}{Spec}

\DeclareMathOperator{\gr}{gr}
\DeclareMathOperator{\DR}{DR}

\DeclareMathOperator{\Pic}{Pic}

\newcommand{\shf}[1]{\mathscr{#1}}

\def\overbar#1#2#3{{%
	\setbox0=\hbox{\displaystyle{#1}}%
	\dimen0=\wd0
	\advance\dimen0 by -#2 
	\vbox {\nointerlineskip \moveright #3 \vbox{\hrule height 0.3pt width \dimen0}%
		\nointerlineskip \vskip 1.5pt \box0}%
}}

\newcommand{\shO}{\shf{O}}

\makeatletter
\let\@@seccntformat\@seccntformat
\renewcommand*{\@seccntformat}[1]{%
  \expandafter\ifx\csname @seccntformat@#1\endcsname\relax
    \expandafter\@@seccntformat
  \else
    \expandafter
      \csname @seccntformat@#1\expandafter\endcsname
  \fi
    {#1}%
}
\newcommand*{\@seccntformat@subsection}[1]{%
  \textbf{\csname the#1\endcsname.}
}
\makeatother

\makeatletter
\let\@paragraph\paragraph
\renewcommand*{\paragraph}[1]{%
	\vspace{0.3\baselineskip}%
	\@paragraph{\textit{#1}}%
}
\makeatother

\counterwithin{equation}{subsection}
\counterwithout{subsection}{section}
\counterwithin{figure}{subsection}

\newtheorem{theorem}{Theorem} [subsection]
\newtheorem*{theorem*}{Theorem}
\newtheorem{lemma}[theorem]{Lemma}
\newtheorem*{lemma*}{Lemma}
\newtheorem{corollary}[theorem]{Corollary}
\newtheorem{proposition}[theorem]{Proposition}
\newtheorem*{proposition*}{Proposition}

\newtheorem{variant}[theorem]{Variant}

\theoremstyle{definition}
\newtheorem{definition}[theorem]{Definition}
\newtheorem*{definition*}{Definition}
\newtheorem{remark}[theorem]{Remark}

\newtheorem{question}[theorem]{Question}

\newtheorem{example}[theorem]{Example}
\newtheorem*{example*}{Example}
\newtheorem*{problem*}{Problem}

\theoremstyle{plain}

\newcommand{\theoremref}[1]{\hyperref[#1]{Theorem~\ref*{#1}}}
\newcommand{\lemmaref}[1]{\hyperref[#1]{Lemma~\ref*{#1}}}
\newcommand{\definitionref}[1]{\hyperref[#1]{Definition~\ref*{#1}}}
\newcommand{\propositionref}[1]{\hyperref[#1]{Proposition~\ref*{#1}}}
\newcommand{\conjectureref}[1]{\hyperref[#1]{Conjecture~\ref*{#1}}}
\newcommand{\corollaryref}[1]{\hyperref[#1]{Corollary~\ref*{#1}}}
\newcommand{\exampleref}[1]{\hyperref[#1]{Example~\ref*{#1}}}

\makeatletter
\let\old@caption\caption
\renewcommand*{\caption}[1]{%
	\setcounter{figure}{\value{equation}}%
	\stepcounter{equation}%
	\old@caption{#1}\relax%
}
\makeatother

\newcounter{intro}

\newtheorem{intro-conjecture}[intro]{Conjecture}
\newtheorem{intro-corollary}[intro]{Corollary}
\newtheorem{intro-theorem}[intro]{Theorem}

\newcommand{\parref}[1]{\hyperref[#1]{\S\ref*{#1}}}

\makeatletter
\newcommand*\if@single[3]{%
  \setbox0\hbox{${\mathaccent"0362{#1}}^H$}%
  \setbox2\hbox{${\mathaccent"0362{\kern0pt#1}}^H$}%
  \ifdim\ht0=\ht2 #3\else #2\fi
  }
\newcommand*\rel@kern[1]{\kern#1\dimexpr\macc@kerna}
\newcommand*\widebar[1]{\@ifnextchar^{{\wide@bar{#1}{0}}}{\wide@bar{#1}{1}}}
\newcommand*\wide@bar[2]{\if@single{#1}{\wide@bar@{#1}{#2}{1}}{\wide@bar@{#1}{#2}{2}}}
\newcommand*\wide@bar@[3]{%
  \begingroup
  \def\mathaccent##1##2{%
    \if#32 \let\macc@nucleus\first@char \fi
    \setbox\z@\hbox{$\macc@style{\macc@nucleus}_{}$}%
    \setbox\tw@\hbox{$\macc@style{\macc@nucleus}{}_{}$}%
    \dimen@\wd\tw@
    \advance\dimen@-\wd\z@
    \divide\dimen@ 3
    \@tempdima\wd\tw@
    \advance\@tempdima-\scriptspace
    \divide\@tempdima 10
    \advance\dimen@-\@tempdima
    \ifdim\dimen@>\z@ \dimen@0pt\fi
    \rel@kern{0.6}\kern-\dimen@
    \if#31
      \overline{\rel@kern{-0.6}\kern\dimen@\macc@nucleus\rel@kern{0.4}\kern\dimen@}%
      \advance\dimen@0.4\dimexpr\macc@kerna
      \let\final@kern#2%
      \ifdim\dimen@<\z@ \let\final@kern1\fi
      \if\final@kern1 \kern-\dimen@\fi
    \else
      \overline{\rel@kern{-0.6}\kern\dimen@#1}%
    \fi
  }%
  \macc@depth\@ne
  \let\math@bgroup\@empty \let\math@egroup\macc@set@skewchar
  \mathsurround\z@ \frozen@everymath{\mathgroup\macc@group\relax}%
  \macc@set@skewchar\relax
  \let\mathaccentV\macc@nested@a
  \if#31
    \macc@nested@a\relax111{#1}%
  \else
    \def\gobble@till@marker##1\endmarker{}%
    \futurelet\first@char\gobble@till@marker#1\endmarker
    \ifcat\noexpand\first@char A\else
      \def\first@char{}%
    \fi
    \macc@nested@a\relax111{\first@char}%
  \fi
  \endgroup
}
\makeatother

\newcommand{\I}{\mathcal{I}}

\def\ZZ{{\mathbf Z}}

\def\CC{{\mathbf C}}
\def\AAA{{\mathbf A}}
\def\RR{{\mathbf R}}
\def\QQ{{\mathbf Q}}
\def\PP{{\mathbf P}}
\def\Pic{{\rm Pic}}

\begin{document}
%\hfill \today

\vspace{\baselineskip}

\title{Hodge ideals for $\QQ$-divisors: birational approach}

\author[M. Musta\c{t}\v{a}]{Mircea~Musta\c{t}\u{a}}
\address{Department of Mathematics, University of Michigan,
Ann Arbor, MI 48109, USA}
\email{{\tt mmustata@umich.edu}}

\author[M.~Popa]{Mihnea~Popa}
\address{Department of Mathematics, Northwestern University, 
2033 Sheridan Road, Evanston, IL
60208, USA} \email{{\tt mpopa@math.northwestern.edu}}

\thanks{MM was partially supported by NSF grant DMS-1701622; MP was partially supported by NSF grant
DMS-1700819.}


\begin{abstract}
We develop the theory of Hodge ideals for $\QQ$-divisors by means of log resolutions, extending our 
previous work on reduced hypersurfaces. We prove local (non-)triviality criteria and a global vanishing 
theorem, as well as other analogues of standard results from the theory of multiplier ideals, and we derive 
a new local vanishing theorem. The connection with the $V$-filtration is analyzed in a sequel.
\end{abstract}

\maketitle

\makeatletter
\newcommand\@dotsep{4.5}
\def\@tocline#1#2#3#4#5#6#7{\relax
  \ifnum #1>\c@tocdepth % then omit
  \else
    \par \addpenalty\@secpenalty\addvspace{#2}%
    \begingroup \hyphenpenalty\@M
    \@ifempty{#4}{%
      \@tempdima\csname r@tocindent\number#1\endcsname\relax
    }{%
      \@tempdima#4\relax
    }%
    \parindent\z@ \leftskip#3\relax
    \advance\leftskip\@tempdima\relax
    \rightskip\@pnumwidth plus1em \parfillskip-\@pnumwidth
    #5\leavevmode\hskip-\@tempdima #6\relax
    \leaders\hbox{$\m@th
      \mkern \@dotsep mu\hbox{.}\mkern \@dotsep mu$}\hfill
    \hbox to\@pnumwidth{\@tocpagenum{#7}}\par
    \nobreak
    \endgroup
  \fi}
\def\l@section{\@tocline{1}{0pt}{1pc}{}{\bfseries}}
\def\l@subsection{\@tocline{2}{0pt}{25pt}{5pc}{}}
\makeatother

\tableofcontents

\section{Introduction}
In this paper we continue the study of Hodge ideals initiated in \cite{MP1}, \cite{MP2}, by considering an analogous theory 
for arbitrary $\QQ$-divisors. The emphasis here is on a birational definition and study of Hodge ideals, while the companion 
paper \cite{MP3} is devoted to a study based on their connection with the $V$-filtration, inspired by \cite{Saito-MLCT}. Both 
approaches turn out to provide crucial information towards a complete understanding of these objects.

Let $X$ be a smooth complex variety. If $D$ is reduced divisor on $X$, the Hodge ideals $I_k (D)$, with $k \ge 0$, 
are defined in terms of the Hodge filtration 
on the $\Dmod_X$-module $\shO_X(*D)$ of functions with poles of arbitrary order along $D$. Indeed, this 
$\Dmod_X$-module
underlies a mixed Hodge module on $X$, and therefore comes with a Hodge filtration $F_\bullet \shO_X (*D)$, which satisfies
$$F_k \shO_X(*D) = I_k (D) \otimes \shO_X\big( (k+1)D\big), \,\,\,\,\,\,{\rm for~all}\,\,\,\, k \ge 0.$$
See \cite{MP1} for details, and for an extensive study of the ideals $I_k (D)$.

Our goal here is to provide a similar construction and study in the general case. A natural device for dealing with the 
fact that fractional divisors are not directly related to Hodge theory is to use new objects derived from covering constructions.
Let $D$ be an arbitrary effective $\QQ$-divisor on $X$. Locally, we can write $D = \alpha H$, for 
some $\alpha \in \QQ_{> 0}$ and $H = {\rm div} (h)$, the divisor of a nonzero regular function; we also denote by $Z$ the support of $D$.
A well-known construction associates to this data a twisted version of the localization $\Dmod$-module above, namely
$$\Mmod (h^{-\alpha}) : = \shO_X (*Z) h^{-\alpha},$$
that is the rank $1$ free $\shO_X(*Z)$-module 
with generator the symbol $h^{-\alpha}$, on which a derivation $D$ of $\shO_X$ acts by
$$D(wh^{-\alpha}) :=\left(D(w) - \alpha w\frac{ D(h)}{h}\right)h^{-\alpha}.$$
It turns out that this $\Dmod_X$-module can be endowed with a natural filtration $F_k \Mmod (h^{-\alpha})$, with 
$k \ge 0$, which makes it a filtered direct summand of a $\Dmod$-module underlying a mixed Hodge module on $X$; see 
\S\ref{coverings}. This plays a role 
analogous to the Hodge filtration, and just as in the reduced case one can show that $F_k \Mmod (h^{-\alpha}) \subseteq 
\shO_X (kZ ) h^{-\alpha}$. This is done in \S\ref{smooth} and \S\ref{definition}, by first analyzing the case when $Z$ is a smooth divisor 
(in this case, if $\lceil D\rceil=Z$, then
the inclusion is in fact an equality). It is therefore natural to define the \emph{$k$-th Hodge ideal} of $D$ by the formula
$$F_k \Mmod (h^{-\alpha}) = I_k (D) \otimes_{\shO_X} \shO_X (kZ ) h^{-\alpha}.$$ 

Similarly to \cite{MP1}, one of our main goals here is to study Hodge ideals of $\QQ$-divisors by means of log resolutions. 
To this end, let $f\colon Y\to X$ be a log resolution of the pair $(X, D)$ that is an isomorphism over $U=X\smallsetminus Z$,  
and denote $g=h\circ f$. There is a filtered isomorphism
$$
\big(\Mmod(h^{-\alpha}), F \big)\simeq f_+\big(\Mmod(g^{-\alpha}), F\big).
$$
Denoting $G=f^*D$ and $E = {\rm Supp}(G)$, so that $E$ is a simple normal crossing divisor, 
it turns out that there exists a complex on $Y$:
$$C^{\bullet}_{g^{-\alpha}} (- \lceil G\rceil):\,\,0\to\shO_Y(-\lceil G\rceil)\otimes_{\shO_Y} \Dmod_Y\to\shO_Y(-\lceil G\rceil)\otimes_{\shO_Y}\Omega^1_Y(\log E)\otimes_{\shO_Y}\Dmod_Y$$
$$\to\ldots\to\shO_Y(-\lceil G\rceil)\otimes_{\shO_Y}\omega_Y(E)\otimes_{\shO_Y}\Dmod_Y\to 0,$$
which is placed in degrees $-n,\ldots,0$, whose differential is described in \S\ref{complex_for_SNC}.
This complex has a natural filtration given, for $k\ge 0$, by subcomplexes 
$$F_{k-n}C^{\bullet}_{g^{-\alpha}} (- \lceil G\rceil): = 
0 \rightarrow \shO_Y(-\lceil G\rceil)\otimes  F_{k-n} \Dmod_Y \rightarrow $$
$$\to  \shO_Y(-\lceil G\rceil)\otimes\Omega_Y^1(\log E) \otimes  F_{k-n+1} \Dmod_Y \rightarrow 
\cdots \to  \shO_Y(-\lceil G\rceil)\otimes \omega_Y(E) \otimes  F_k \Dmod_Y\rightarrow 0.$$

\noindent
Extending \cite[Proposition 3.1]{MP1}, we show in Proposition \ref{prop_SNC_complex} and Proposition \ref{Hodge_ideals_SNC} 
that there is a filtered quasi-isomorphism
$$
\big(C^{\bullet}_{g^{-\alpha}} (- \lceil G\rceil), F \big) \simeq \big(\Mmod_r (g^{-\alpha}), F \big),
$$
where $\Mmod_r (g^{-\alpha})$ is the filtered right $\Dmod_Y$-module associated to $\Mmod (g^{-\alpha})$.
Thus one can use $\big(C^{\bullet}_{g^{-\alpha}} (- \lceil G\rceil), F\big)$ as a concrete representative for computing the filtered 
$\Dmod$-module pushforward of $\big(\Mmod_r (g^{-\alpha}), F\big)$, hence for computing 
the ideals $I_k (D)$. More precisely, we have 
$$R^0f_*F_{k-n}\big(C^{\bullet}_{g^{-\alpha}} (- \lceil G\rceil)\otimes_{\Dmod_Y}\Dmod_{Y\to X}\big)\simeq
h^{-\alpha}\omega_X (kZ )\otimes_{\shO_X}I_k(D).$$
See Theorem \ref{formula_log_resolution} for a complete picture regarding this push-forward operation.

This fact, together with special properties of the filtration on $\Dmod$-modules underlying mixed Hodge modules, leads to 
our main results on Hodge ideals, which are collected in the following:

\begin{intro-theorem}\label{main_list}
In the set-up above, the Hodge ideals $I_k (D)$ satisfy:

%\medskip
\noindent
(i) $I_0 (D)$ is the multiplier ideal $\I \big((1 - \epsilon)D\big)$, 
so in particular $I_0 (D) = \shO_X$ if and only if the pair $(X, D)$ is log canonical; see \S\ref{scn:I_0}.

%\medskip
\noindent
(ii) If $Z$ has simple normal crossings, then 
$$I_k (D) = I_k (Z) \otimes \shO_X (Z - \lceil D \rceil),$$
while $I_k (Z)$ can be computed explicitly as in \cite[Proposition~8.2]{MP1}; see \S\ref{SNC}.  In particular, if 
$Z$ is smooth, then $I_k (D) = \shO_X (Z - \lceil D \rceil)$ for all $k$; cf. also Corollary \ref{smoothness_equivalence}.

%\medskip
\noindent
(iii) The Hodge filtration is generated at level $n-1$, where $n = \dim X$, i.e. 
$$F_\ell \Dmod_X\cdot\big(I_k(D)\otimes \shO_X(kZ)h^{-\alpha}\big) = I_{k+\ell}(D)\otimes \shO_X \big((k+\ell)Z\big)h^{-\alpha}$$
for all $k \ge n-1$ and $\ell\ge 0$; see \S\ref{scn:level}.

%\medskip
\noindent
(iv) There are non-triviality criteria for $I_k (D)$ at a point $x \in D$ in terms of the multiplicity of $D$ at $x$; 
see \S\ref{nontriviality}.

%\medskip
\noindent
(v) If $X$ is projective, $I_k(D)$ satisfy a vanishing theorem analogous to Nadel Vanishing for multiplier ideals; 
see \S\ref{vanishing}.

%\medskip
\noindent
(vi) If $Y$ is a smooth divisor in $X$ such that $Z|_Y$ is reduced, then $I_k (D)$ satisfy 
$$I_k (D|_Y) \subseteq I_k(D) \cdot \shO_Y,$$
with equality when $Y$ is general; see \S\ref{scn:restriction} for a more general statement.

%\medskip
\noindent
(vii) If $ X \to T$ is a smooth family with a section $s\colon T \to X$, and $D$ is a relative divisor on $X$ that satisfies a suitable condition
(see \S\ref{scn:semicontinuity} for the precise statement) then 
$$\{ t \in T ~|~ I_k (D_t) \not\subseteq \frak{m}_{s(t)}^q\}$$
is an open subset of $T$, for each $q \ge 1$.

%\medskip
\noindent
(viii) If $D_1$ and $D_2$ are $\QQ$-divisors with supports $Z_1$ and $Z_2$, such that $Z_1 +  Z_2$ is also reduced, then  the subadditivity property
$$I_k (D_1 + D_2) \subseteq I_k (D_1)\cdot I_k (D_2)$$
holds; see \S\ref{scn:subadditivity} for a more general statement.
\end{intro-theorem}

For comparison, the list of properties of Hodge ideals in the case when $D$ is reduced is summarized in \cite[\S4]{Popa}. While much of the story carries over to the setting of $\QQ$-divisors -- besides of course the connection with the classical Hodge theory of the complement $U = X \smallsetminus D$, which only makes sense in the reduced case -- there are a few significant points where the picture becomes more intricate. For instance, the bounds for the generation level of the Hodge filtration can become worse.  
Moreover, we do not know 
whether the inclusions $I_k (D) \subseteq I_{k-1} (D)$ continue to hold for arbitrary $\QQ$-divisors. New phenomena appear as well: unlike in the case of multiplier ideals, for 
rational numbers $\alpha_1 < \alpha_2$, usually the ideals $I_k (\alpha_1 Z)$ and $I_k (\alpha_2 Z)$ cannot be 
compared for $k \ge 1$; see for instance Example \ref{parameter_dependence}.

It turns out however that most of these issues disappear if one works modulo the ideal of the hypersurface, at least for 
rational multiples of a reduced divisor.
This, as well as other basic facts, is addressed in the sequel \cite{MP3}, which 
studies Hodge ideals from a somewhat different point of view, namely by comparing them to
the (microlocal) $V$-filtration induced on $\shO_X$ by $h$. This is inspired by the work of Saito \cite{Saito-MLCT} in the reduced case. In the statement below we summarize some of these properties, which complement the results in Theorem \ref{main_list}, but which we do not know how to obtain with the methods of this paper.

\begin{intro-theorem}\cite{MP3}\label{other_paper}
Let $D = \alpha Z$, where $Z$ is a reduced divisor and $\alpha \in \QQ_{> 0}$. Then the following hold:
\begin{enumerate}
\item $I_k (D) + \shO_X(-Z) \subseteq I_{k-1} (D) + \shO_X(-Z)$ for all $k$.
\item If $\alpha \in (0 , 1]$, then $I_k (D) = \shO_X \iff k \le \widetilde{\alpha}_Z - \alpha$, where 
$\widetilde{\alpha}_Z$  is the negative of the largest root of the reduced Bernstein-Sato polynomial of $Z$.
\item If $I_{k-1}(D) = \shO_X$ (we say that $(X, D)$ is $(k-1)$-log canonical),
then $I_{k+1} (D) \subseteq I_k (D)$.
\item Fixing $k$, there exists a finite set of rational numbers $0 = c_0 < c_1 < \cdots < c_s < c_{s+1} = 1$ such that 
for each $0 \le i \le s$ and each $\alpha \in (c_i, c_{i+1}]$ we have 
$$I_k (\alpha Z)\cdot\shO_Z = I_k (c_{i+1} Z)\cdot\shO_Z = {\rm constant}$$
and such that 
$$I_k (c_{i+1} Z)\cdot\shO_Z\subsetneq I_k (c_i Z)\cdot\shO_Z.$$
\end{enumerate}
\end{intro-theorem}

Going back to the description of Hodge ideals by means of log resolutions, the strictness of the Hodge filtration 
for the push-forwards of (summands of) mixed Hodge modules leads to the following local Nakano-type vanishing 
result for $\QQ$-divisors:

\begin{intro-corollary}\label{intro_local_vanishing}
Let $D$ be an effective $\QQ$-divisor on a smooth variety $X$ of dimension $n$, and let $f\colon Y\to X$ be a log resolution of $(X,D)$ that is an isomorphism over
$X\smallsetminus {\rm Supp}(D)$. If $E= (f^*D)_{{\rm red}}$, then
$$R^q f_*\big(\Omega_Y^p(\log E)\otimes_{\shO_Y} \shO_Y(-\lceil f^*D\rceil)\big)=0\quad\text{for}\quad p+q>n.$$
\end{intro-corollary}

Note that for $p =n$ this is the local vanishing for multiplier ideals \cite[Theorem 9.4.1]{Lazarsfeld}, since $E - \lceil f^*D\rceil = - [(1- \epsilon) f^*D]$ for $0 < \epsilon \ll 1$. In general, the statement extends the case of reduced $D$ in \cite[Corollary 3]{Saito-LOG} (cf. also \cite[\S{A.5}]{Saito-MLCT}). Unlike \cite[Theorem 32.1]{MP1} regarding that case, at the moment we are unable to prove this corollary via more elementary methods.

A different series of applications, given in \cite{MP3}, uses the results proved in this paper together with the relationship between Hodge ideals of $\QQ$-divisors and the $V$-filtration, in order to describe the behavior of the invariant $\widetilde{\alpha}_Z$ described in Theorem \ref{other_paper} (called the \emph{minimal exponent} of $Z$). For instance, the triviality criterion proved here as Proposition \ref{triviality_criterion} 
leads to a lower bound \cite[Corollary ~D]{MP3} for  $\widetilde{\alpha}_Z$ in 
terms of invariants on a log resolution, addressing a question of Lichtin and Koll\'{a}r. Moreover, the results
in Theorem \ref{main_list} (vi) and (vii), and Corollary \ref{good_bound}, lead to effective bounds and to restriction and semicontinuity statements for $\widetilde{\alpha}_Z$, in analogy with well-known properties of log canonical thresholds; for details see \cite[\S6]{MP3}.

\section{Hodge ideals via log resolutions, and first properties}

Let $X$ be a smooth complex algebraic variety.
Given an effective $\QQ$-divisor $D$ on $X$, our goal is to attach to $D$ ideal sheaves $I_k(D)$ for $k \ge 0$; when $D$ is a reduced divisor, these will coincide with the Hodge ideals in \cite{MP1}.

\subsection{A brief review of Hodge modules}\label{Hodge_modules}
A key ingredient for the definition of our invariants is Saito's theory of mixed Hodge modules. In what follows,
we give a brief presentation of the relevant objects, and recall a few facts that we will need.
For details, we refer to \cite{Saito-MHM}.

Given a smooth $n$-dimensional complex algebraic variety $X$, we denote by $\Dmod_X$ the sheaf of differential operators on $X$. 
This carries the increasing filtration $F_{\bullet}\Dmod_X$ by order of differential operators. A \emph{left} or \emph{right $\Dmod$-module}
is a left, respectively right, $\Dmod_X$-module, which is quasi-coherent as an $\shO_X$-module. There is an equivalence 
between the categories of left and right $\Dmod$-modules, which at the level of $\shO_X$-modules is given by
$$\Mmod\to \Nmod:=\Mmod\otimes_{\shO_X}\omega_X\quad\text{and}\quad \Nmod\to {\mathcal Hom}_{\shO_X}(\omega_X,\Nmod).$$
For example, this equivalence maps the left $\Dmod$-module $\shO_X$ to the right $\Dmod$-module $\omega_X$.
For a thorough introduction to the theory of $\Dmod$-modules, we refer to \cite{HTT}.

A \emph{filtered left} (or \emph{right}) \emph{$\Dmod$-module} 
is a $\Dmod$-module $\Mmod$, together with an increasing filtration $F=F_{\bullet}\Mmod$ that is compatible with the
order filtration on $\Dmod_X$ and which is \emph{good}, in a sense to be defined momentarily. A morphism of filtered $\Dmod$-modules is required to be compatible with the filtrations. The equivalence between left
and right $\Dmod$-modules extends to the categories of filtered modules, with the convention that 
$$F_{p-n}(\Mmod\otimes_{\shO_X}\omega_X)=F_p\Mmod\otimes_{\shO_X}\omega_X.$$
A filtration $F_{\bullet}\Mmod$ on a coherent $\Dmod$-module $\Mmod$ is \emph{good} if the corresponding graded object 
${\rm gr}_{\bullet}^F\Mmod:=\bigoplus_kF_k\Mmod/F_{k-1}\Mmod$ is locally finitely generated over ${\rm gr}_{\bullet}^F\Dmod_X$. 
We note that
every coherent $\Dmod$-module admits a good filtration, but this is far from being unique.

We now come to the key objects in Saito's theory, the \emph{mixed Hodge modules} from \cite{Saito-MHM}. Such an object is given by the data 
$M=(\Mmod,F,\Pmod ,\varphi, W)$, where:
\begin{enumerate}
\item[i)] $(\Mmod,F)$ is a filtered $\Dmod$-module, with $\Mmod$ a holonomic left (or right) $\Dmod$-module, with regular singularities; 
$F$ is the \emph{Hodge filtration} of $\Mmod$.
\item[ii)] $\Pmod$ is a perverse sheaf of $\QQ$-vector spaces on $X$.
\item[iii)] $\varphi$ is an isomorphism between $\Pmod_{\CC} = \Pmod \otimes_{\QQ} \CC$ and ${\rm DR}(\Mmod)$, i.e. 
the perverse sheaf corresponding to $\Mmod$ via the Riemann-Hilbert correspondence.
\item[iv)] $W$ is a finite, increasing filtration on $(\Mmod,F, \Pmod,\varphi)$, the \emph{weight filtration} of the mixed Hodge module.
\end{enumerate}
For a such an object to be a mixed Hodge module, it has to satisfy 
a complicated set of conditions of an inductive nature, which we do not discuss here.
The main reference for the basic definitions and results of this theory is \cite{Saito-MHM}; see also 
\cite{Saito-YPG} for an introduction.

Given a mixed Hodge module $(\Mmod,F,\Pmod,\varphi, W)$, we say that the filtered $\Dmod$-module $(\Mmod,F)$ is a \emph{Hodge $\Dmod$-module}
(or that it
\emph{underlies} a mixed Hodge module). In fact, this is the only piece of information that we will be concerned with in this article. 
The basic example of a mixed Hodge module is ${\mathbf Q}_X^H[n]$, the trivial one. In this case, the filtered $\Dmod$-module is the structure sheaf $\shO_X$, 
with the filtration such that ${\rm gr}_p^F\shO_X=0$ for all $p\neq 0$. The corresponding perverse sheaf is $\QQ_X [n]$ and the weight filtration is such that 
${\rm gr}_p^W\shO_X=0$ for $p\neq n$. 

The mixed Hodge modules on $X$ form an Abelian category, denoted ${\rm MHM}(X)$. Morphisms in this category are strict with respect to both the Hodge
and the weight filtration. The corresponding bounded derived category is denoted ${\bf D}^b\big({\rm MHM}(X)\big)$. 

Mixed Hodge modules satisfy Grothendieck's 6 operations formalism. The relevant fact for us is that to every morphism $f\colon X\to Y$ of 
smooth complex algebraic varieties we have a corresponding push-forward functor $f_+\colon {\bf D}^b\big({\rm MHM}(X)\big)\to {\bf D}^b\big({\rm MHM}(Y)\big)$
(this is denoted by $f_*$ in \cite{Saito-MHM}). Moreover, if
$g\colon Y\to Z$ is another such morphism, we have a functorial isomorphism $(g\circ f)_+\simeq g_+\circ f_+$. 

Regarding the push-forward functor for mixed Hodge modules, we note that on the level of $\Dmod$-modules, it coincides with the usual $\Dmod$-module push-forward. 
Moreover, if $f\colon X\to Y$ is proper and if we denote by ${\rm FM}(\Dmod_X)$ the category of filtered $\Dmod$-modules on $X$ (here it is convenient to work with right $\Dmod$-modules), then
Saito defined in \cite{Saito-MHP} a functor
$$f_+ \colon {\bf D}^b \big({\rm FM}(\Dmod_X)\big) \rightarrow {\bf D}^b \big({\rm FM}(\Dmod_Y)\big).$$
This is compatible with the usual direct image functor for right $\Dmod$-modules and it is used to define
the push-forward between the derived categories of mixed Hodge modules at the level of filtered complexes. 
With a slight abuse of notation, if $(\Mmod,F)$ underlies a mixed Hodge module $M$ on $X$ and if $f\colon X\to Y$
is an arbitrary morphism, then we write $f_+(\Mmod,F)$ for the object in ${\bf D}^b\big({\rm FM}(\Dmod_Y)\big)$ underlying $f_+M$.

An important feature of the push-forward of Hodge $\Dmod$-modules with respect to proper morphisms is \emph{strictness}. 
This says that if $f\colon X\to Y$ is proper and $(\Mmod,F)$ underlies a mixed Hodge module on $X$, then 
$f_+ (\Mmod, F)$ is strict 
as an object in ${\bf D}^b \big({\rm FM}(\Dmod_Y)\big)$ (and moreover, each $H^i f_+ (\Mmod, F)$ underlies a Hodge $\Dmod_Y$-module). 
This means that the natural mapping 
\begin{equation}\label{strictness_formula}
R^i f_* \big(F_k (\Mmod \overset{\derL}{\otimes}_{\Dmod_X} \Dmod_{X\to Y}) \big) \longrightarrow 
R^i f_* (\Mmod \overset{\derL}{\otimes}_{\Dmod_X} \Dmod_{X\to Y})
\end{equation}
is injective for every $i, k\in\ZZ$. 
Taking $F_kH^if_+(\Mmod,F)$ to be the image of this map, we get the filtration on $H^if_+(\Mmod,F)$.

The push-forward with respect to open embeddings is more subtle. For example, suppose that $Z$ is an effective divisor
on the smooth variety $X$ and $j\colon U=X\smallsetminus Z\hookrightarrow X$ is the corresponding open immersion. 
Recall that $\shO_X(*Z)$
is the push-forward $j_*\shO_U$; on a suitable affine open neighborhood $V$ of a given point in $X$, this is given by localizing $\shO_X(V)$ at 
an equation defining $Z\cap V$ in $V$.
$\shO_X(*Z)$ has a natural left $\Dmod$-module structure induced by the canonical $\Dmod$-module structure on $\shO_X$. In fact, as such
we have $\shO_X(*Z)\simeq j_+\shO_U$  (in general, for a $\Dmod_U$-module $\Mmod$, the $\Dmod$-module push-forward
 $j_+\Mmod$ agrees with $j_*\Mmod$, with the induced $\Dmod_X$-module structure).
We thus see that $\shO_X(*Z)$ carries a canonical filtration such that the corresponding
filtered $\Dmod$-module underlies $j_+\QQ_U^H[n]$. This filtration is the one that leads to the Hodge ideals studied in \cite{MP1}.

\subsection{Filtered $\Dmod$-modules associated to $\QQ$-divisors}\label{coverings}
Let $X$ be a smooth complex algebraic variety, with $\dim(X)=n$. The ideals we associate to effective $\QQ$-divisors on $X$ arise from certain Hodge $\Dmod$-modules. 
The $\Dmod$-modules themselves have been extensively studied: these are the 
$\Dmod$-modules attached to rational powers of functions on $X$. We proceed to recall their definition.

Consider a nonzero $h\in\shO_X(X)$ and $\beta\in\QQ$. We denote by $Z$ the reduced divisor on $X$ with the same support as
$H={\rm div}(h)$ and let $j\colon U=X\smallsetminus {\rm Supp}(Z)\hookrightarrow X$ be the inclusion map. 
 We consider the left $\Dmod_X$-module 
$\Mmod(h^{\beta})$, which is a rank 1 free $\shO_X(*Z)$-module 
with generator the symbol $h^{\beta}$, on which a derivation $D$ of $\shO_X$ acts by
$$D(wh^{\beta}) :=\left(D(w) +  w\frac{\beta\cdot D(h)}{h}\right)h^{\beta}.$$
We will denote the corresponding right $\Dmod_X$-module by $\Mmod_r(h^{\beta})$.
This can be described as $h^{\beta}\omega_X(*Z)$, an $\shO_X$-module isomorphic to $\omega_X(*Z)$,
and such that if $x_1,\ldots,x_n$ are local coordinates, then 
$$(h^{\beta}w dx_1\cdots dx_n)\partial_i=- h^{\beta}\left(\frac{\partial w}{\partial x_i}+ w\frac{\beta}{h}\cdot \frac{\partial h}{\partial x_i}\right)dx_1\cdots dx_n$$
for every $i$ with $1\leq i\leq n$.

\begin{remark}\label{case_ell_1}
When $\beta\in\ZZ$, we have a canonical isomorphism of left $\Dmod_X$-modules
\begin{equation}\label{eq_case_ell_1}
\Mmod(h^{\beta})\simeq \shO_X(*Z),\,\,\,\,\,\,wh^{\beta}\to w \cdot h^{\beta},
\end{equation}
where on the localization $\shO_X(*Z)$ we consider the natural $\Dmod_X$-module structure induced from $\shO_X$. Note that $\shO_X(*Z)$
is also the $\Dmod$-module push-forward $j_+ \shO_U$.
\end{remark}

\begin{remark}\label{eq_renormalization_0}
For every positive integer $m$, we have a canonical isomorphism of  left $\Dmod_X$-modules
$$\Mmod(h^{\beta})\simeq\Mmod\big((h^m)^{\beta /m}\big),\,\,\,\,\,\,wh^{\beta}\to w(h^m)^{\beta/m}.$$
\end{remark}

\begin{remark}\label{twist_individual}
We can define, more generally, left $\Dmod$-modules $\Mmod(h_1^{\beta_1}\cdots h_r^{\beta_r})$, for 
nonzero regular functions $h_1,\ldots,h_r\in\shO_X(X)$ and rational numbers $\beta_1,\ldots,\beta_r$. 
If $\ell_i$ are positive integers such that
$\beta_i/\ell_i=\beta$ for all $i$ and if $h=\prod_ih_i^{\ell_i}$, then we have an isomorphism of left $\Dmod_X$-modules
$$\Mmod(h_1^{\beta_1}\cdots h_r^{\beta_r})\simeq \Mmod(h^{\beta}).$$
\end{remark}

\begin{remark}\label{twist_positive}
If $r$ is an integer, then we have an isomorphism of left $\Dmod_X$-modules
$$\Mmod(h^{\beta})\to \Mmod(h^{r + \beta}), \,\,\,\,\,\,wh^{\beta}\to (wh^{-r}) h^{r+\beta}.$$
\end{remark}

\medskip

Let now $D$ be an effective $\QQ$-divisor on $X$. We denote by $Z$ the reduced divisor with the same support as $D$. 
As above, we put $U=X\smallsetminus Z$ and let $j\colon U\hookrightarrow X$ be the inclusion map. 
We first assume that we can write $D=\alpha\cdot {\rm div}(h)$ for some nonzero $h\in\shO_X(X)$ and $\alpha\in\QQ_{>0}$ (this is of course 
always the case locally). To this data we can associate the $\Dmod_X$-module $\Mmod(h^{-\alpha})$; later it will be more convenient to 
consider equivalently (according to Remark \ref{twist_positive}) the $\Dmod_X$-module $\Mmod(h^{1-\alpha})$. 
This depends on the choice of $h$; however, if we replace $h$ by $h^m$ and $\alpha$ by $\alpha/m$, for some positive integer $m$,
the $\Dmod$-module does not change (see Remark~\ref{eq_renormalization_0}). In particular, we may always assume that $\alpha=1/\ell$,
for a positive integer $\ell$.

\begin{remark}\label{rmk_add_eff_div}
Suppose that $D'$ is a $\QQ$-divisor with the same support as $D$ and such that $D-D'={\rm div}(u)$, for some $u\in\shO_X(X)$.
Suppose that we can write
$D'=\frac{1}{\ell}\cdot {\rm div}(h')$ for some $h'\in\shO_X(X)$ and some positive integer $\ell$.
In this case we can also write
$D=\frac{1}{\ell}\cdot {\rm div}(h)$, where $h=u^{\ell}{h'}$, and  we
have an isomorphism of  $\Dmod_X$-modules
\begin{equation}\label{isom_add_1}
\Mmod(h^{-1/\ell})\to \Mmod({h'}^{-1/\ell}), \,\,\,\,\,\,gh^{-1/\ell} \to gu^{-1}{h'}^{-1/\ell}.
\end{equation}
\end{remark}

\medskip

Our first goal is to show that $\Mmod(h^{-\alpha})$ is canonically a filtered $\Dmod_X$-module. 
Let $\ell$ be a positive integer such that $\ell\alpha\in\ZZ$.
Consider the finite \'{e}tale map $p\colon V\to U$, where $V={\bf Spec}\,\shO_U[y]/(y^{\ell}-h^{-\ell\alpha})$.
Note that this fits in a Cartesian diagram
\begin{equation}\label{cart_diag1}
\begin{tikzcd}
V \rar\dar{p} & W\dar{q} \\
U\rar{j} & X,
\end{tikzcd}
\end{equation}
in which
$$W={\bf Spec}\,\shO_X[z]/(z^\ell-h^{\ell\alpha}),$$
such that the map $V\to W$ pulls $z$ back to $y^{-1}=y^{\ell-1}h^{\ell\alpha}$.

\begin{lemma}\label{lem_decomp1}
We have an isomorphism of left $\Dmod_X$-modules
\begin{equation}\label{eq1_lem_decomp1}
j_+p_+\shO_V\simeq\bigoplus_{i=0}^{\ell-1}\Mmod(h^{-i\alpha}),
\end{equation}
with the convention that the first summand is $\shO_X(*Z)$.
\end{lemma}

\begin{proof}
Since $p$ is finite \'{e}tale, it follows that we have a canonical isomorphism $\tau\colon p^*\Dmod_U\simeq\Dmod_V$, 
and for every $\Dmod_V$-module $\Mmod$ we have $p_+\Mmod\simeq p_*\Mmod$,
with the action of $\Dmod_U$ induced via the isomorphism $\tau$.

By mapping $gy^i$ to $gh^{-i\alpha}$, where $g$ is a section of $\shO_X$ and $0\leq i\leq\ell-1$, we obtain an isomorphism 
of $\shO_X$-modules as in
(\ref{eq1_lem_decomp1}). In order to see that this is an isomorphism of $\Dmod_X$-modules, consider a local derivation
$D$ of $\shO_X$ and note that since $y^{\ell}=h^{-\ell\alpha}$, by identifying $D$ with its pull-back to $V$ we have
$$D (y^i)=iy^{i-1}D(y)=-i\alpha y^i\frac{D(h)}{h},$$
which via our map corresponds to $D (h^{-i\alpha})$.
This implies the assertion.
\end{proof}

It follows from the lemma that the right-hand side of (\ref{eq1_lem_decomp1}) is the $\Dmod$-module corresponding to the 
mixed Hodge module push-forward $(j\circ p)_+\QQ_V^H[n]$. In particular, it 
 carries a canonical structure of filtered $\Dmod$-module.

\begin{remark}\label{rem_replace_by_multiple}
Let's see what happens if we replace $\ell$ by a multiple $m\ell$. Let $p_{\ell}\colon V_{\ell}\to U$ and $p_{m\ell}\colon V_{m\ell}\to U$ be the corresponding \'{e}tale covers.
Note that $$V_{m\ell}={\bf Spec}\,\shO_U[y]/(y^{\ell m}-h^{-\ell m\alpha})$$ decomposes as a disjoint union of $m$ copies of $V_{\ell}$, and thus we have an isomorphism of filtered $\Dmod_X$-modules (and a corresponding isomorphism of mixed Hodge modules)
\begin{equation}\label{eq_rem_replace_by_multiple}
j_+(p_{m\ell})_+\shO_{V_{m\ell}}\simeq 
\big(j_+(p_{\ell})_+\shO_{V_{\ell}}\big)^{\oplus m}.
\end{equation}
If $\eta$ is a primitive root of $1$ of order $\ell m$, and if on each side of 
(\ref{eq_rem_replace_by_multiple}) we consider the decompositions 
(\ref{eq1_lem_decomp1}), then the isomorphism maps 
$$h^{-i\alpha} \to \big(\eta^{is}h^{-c\ell\alpha}\cdot h^{-d\alpha}\big)_{0\leq s\leq m-1},$$
where we write $i=\ell c+d$, with $0\leq c\leq m-1$ and $0\leq d\leq\ell-1$.
\end{remark}

We can interpret the isomorphism in (\ref{eq1_lem_decomp1}) in terms of a suitable $\mu_\ell$-action, where $\mu_{\ell}$
is the group of $\ell$-th roots of $1$ in $\CC^*$. Note that we have a natural action of $\mu_{\ell}$ on $W$ such that
via the corresponding action on $\shO_W$, an element $\lambda\in\mu_{\ell}$ maps $z^i$ to $\lambda^iz^i$. 
If we let $\mu_{\ell}$ act trivially on $X$, then $q$ is an equivariant morphism (in fact, $q$ is the quotient morphism with respect
to the $\mu_{\ell}$-action). It is clear that $q^{-1}(Z)$ is fixed by the $\mu_{\ell}$-action and we have an induced $\mu_{\ell}$-action
on $W\smallsetminus q^{-1}(Z)=V$. This in turn induces a $\mu_{\ell}$-action on $j_+p_+\shO_V$ and the isomorphism in
(\ref{eq1_lem_decomp1}) corresponds to the isotypic decomposition of $j_+p_+\shO_V$, such that every 
$\lambda\in\mu_{\ell}$ acts on $\Mmod(h^{-i\alpha})$ by multiplication with $\lambda^{-i}$.

\begin{lemma}\label{decomp_filtration}
The filtration on $j_+p_+\shO_V$ is preserved by the $\mu_{\ell}$-action. Therefore we have an induced filtration on each 
$\Mmod(h^{-i\alpha})$ such that (\ref{eq1_lem_decomp1}) is an isomorphism of filtered $\Dmod$-modules.
\end{lemma}

\begin{proof}
One way to see this is by using a suitable equivariant resolution of $W$. Let $W'$ be the disjoint union of the irreducible components of $W$ and
$q'\colon W'\to W$ the canonical morphism. It is clear that the $\mu_{\ell}$-action on $W$ induces an action on $W'$ such that $q'$ is equivariant.
Since $V$ is contained in the smooth locus of $W$, it has an open immersion into $W'$. 
We use equivariant resolution of singularities to construct a $\mu_{\ell}$-equivariant morphism 
$\varphi\colon Y\to W'$ that is an isomorphism over $V$ and such that $(q\circ q'\circ\varphi)^*(Z)$ is a divisor with simple normal crossings.  
Let $g=q\circ q'\circ \varphi$. If $E$ is the reduced, effective divisor supported on $g^{-1}(Z)$, then 
we have an isomorphism of filtered $\Dmod$-modules (induced by a corresponding isomorphism of mixed Hodge modules)
\begin{equation}\label{eq_ism_filtered}
j_+p_+\shO_V\simeq g_+\widetilde{j}_+\shO_V\simeq g_+\shO_Y(*E),
\end{equation}
where $\widetilde{j}\colon Y\smallsetminus {\rm Supp}(E)\hookrightarrow Y$ is the inclusion map.

We can deduce the assertion in the lemma from an explicit
computation of the filtration on $j_+p_+\shO_V$ via the isomorphism (\ref{eq_ism_filtered}), as follows. First, since we deal with
$\Dmod$-module push-forward, it is more convenient to work with right $\Dmod$-modules. We will thus compute $g_+\omega_Y(*E)$, where
$\omega_Y(*E)$ is the filtered right $\Dmod$-module corresponding to $\shO_Y(*E)$. 

Since $E$ is a simple normal crossing divisor, $\omega_Y(*E)$ has a resolution by a complex $C^\bullet$ 
of filtered right $\Dmod_Y$-modules
$$0\longrightarrow C^{-n}\longrightarrow\cdots\longrightarrow C^0\longrightarrow 0,$$
where $C^{i}=\Omega_Y^{i+n}(\log E)\otimes_{\shO_Y}\Dmod_Y$, with the filtration given by
$$F_{k-n}C^i=\Omega_Y^{i+n}(\log E)\otimes_{\shO_Y}F_{k+i}\Dmod_Y.$$ 
For a description of the maps in this complex, see the beginning of \S\ref{complex_for_SNC} below;
a proof
of the fact that it resolves $\omega_Y(*E)$ is given in \cite[Proposition~3.1]{MP1}. We can thus compute  $F_kg_+\omega_Y(*E)$ as the image of the injective map
$$R^0g_*\big(F_k(C^{\bullet}\otimes_{\Dmod_Y}\Dmod_{Y\to X})\big)\to R^0g_*(C^{\bullet}\otimes_{\Dmod_Y}\Dmod_{Y\to X})=g_+\omega_Y(*E).$$
Since $g$ is equivariant and the action of $\mu_{\ell}$ on $Y$ induces an action on $E$ (in fact, it fixes $E$), the above description implies that each
$F_kg_+\omega_Y(*E)$ is preserved by the $\mu_{\ell}$-action.
\end{proof}

\begin{remark}\label{rmk_filtration_first_piece}
We note that the filtration on $j_+p_+\shO_V$ induces the canonical filtration on the first summand $\shO_X(*Z)$.
Indeed, on $U$ we have a morphism of mixed Hodge modules $\QQ_U^H[n]\to p_+\QQ_V^H[n]$. Applying $j_+$ and only considering the underlying filtered $\Dmod$-modules, we obtain a morphism $j_+\shO_U\to j_+p_+\shO_V$, which is an isomorphism onto the first summand.
\end{remark}

\begin{definition}
Given $\alpha>0$, choose $\ell\geq 2$ such that $\ell\alpha\in\ZZ$. In this case $\Mmod(h^{-\alpha})$ appears as the second summand in the decomposition
(\ref{eq1_lem_decomp1}). We define the filtration 
$$F_k \Mmod(h^{-\alpha}) \,\,\,\,\,\,{\rm for} \,\,\,\,\,\, k \ge 0$$ 
to be the filtration induced from the canonical filtration on $j_+p_+\shO_V$. 
It is straightforward to see, using the discussion in Remark~\ref{rem_replace_by_multiple} that this filtration does not change if we replace $\ell$
by a multiple; therefore it is independent of $\ell$. Moreover, we note that if $\alpha$ is an integer, using the same Remark~\ref{rem_replace_by_multiple}, 
the isomorphism $\Mmod(h^{-\alpha})\simeq\shO_X(*Z)$ is an isomorphism of filtered $\Dmod$-modules.
\end{definition}

In this definition, a priori different covers have to be considered for each of the summands $\Mmod(h^{-i\alpha})$. However, 
we have:

\begin{lemma}\label{lem_decomp1_v2}
With the filtration defined above, the isomorphism (\ref{eq1_lem_decomp1}) is an isomorphism of filtered $\Dmod$-modules. 
\end{lemma}
\begin{proof}
By Lemma~\ref{decomp_filtration}, we only need to show that for every $i$ with $0\leq i\leq \ell-1$, the filtration induced on $\Mmod(h^{-i\alpha})$ by that on $j_+p_+\shO_V$ coincides with the 
one given in the above definition. For $i=0$ this follows from Remark~\ref{rmk_filtration_first_piece}. If $i>0$, consider the cover used to define the filtration on 
$\Mmod(h^{-i\alpha})$, namely 
$$p'\colon V'={\bf Spec}\shO_U[y]/(y^{\ell}-h^{i\alpha\ell})\to U.$$ 
Note that we have a finite morphism
$\psi\colon V\to V'$ of varieties over $U$, that pulls-back $y$ to $y^i$. We have a canonical morphism of mixed Hodge modules
$\QQ_{V'}^H[n]\to \psi_+\QQ_V^H[n]$. Applying $j_+p'_+$ and passing to the underlying filtered $\Dmod$-modules, we obtain a morphism of filtered 
$\Dmod$-modules $j_+p'_+\shO_{V'}\to j_+p_+\shO_V$ that is the identity on the summand $\Mmod(h^{-i\alpha})$. This proves our claim.
\end{proof}

\begin{remark}\label{same_filtration_rescaling}
It is clear from definition that for every $\alpha>0$ and every positive integer $m$, the isomorphism
$$\Mmod(h^{-\alpha})\to \Mmod\big((h^m)^{-\alpha/m}\big), \quad gh^m\to g (h^m)^{-\alpha/m}$$
is an isomorphism of filtered $\Dmod$-modules. 
\end{remark}

\begin{remark}\label{rem_isom_comp_with_filtrations}
In the setting of Remark~\ref{rmk_add_eff_div},
the isomorphism (\ref{isom_add_1}) is an isomorphism of filtered $\Dmod_X$-modules.
This is clear if $\ell=1$, hence we assume $\ell\geq 2$.
Let
$p\colon V\to U$ and $p'\colon V'\to U$ be the canonical projections, where
$$V={\bf Spec}\,\shO_U[y]/(y^{\ell}-h)\quad\text{and}
\quad V'={\bf Spec}\,\shO_U[y]/(y^{\ell}-h').$$
We have an isomorphism $\varphi\colon V'\to V$ of schemes over $U$, where $\varphi^*(y)=uy$. 
This induces an isomorphism of filtered $\Dmod_X$-modules
$$j_+p_+\shO_V\simeq j_+p'_+\shO_{V'},$$
which via the identifications given by Lemma~\ref{lem_decomp1} is the direct sum
$$\bigoplus_{i=0}^{\ell-1}\Mmod(h^{-i/\ell})\simeq 
\bigoplus_{i=0}^{\ell-1}\Mmod({h'}^{-i/\ell})$$
of the isomorphisms (\ref{isom_add_1}). For $i=1$, we obtain our assertion.
\end{remark}

A special case of the above remark implies that for every $\alpha>0$ the isomorphism
$$\Mmod(h^{-\alpha})\to\Mmod(h^{-\alpha-1}),\quad gh^{-\alpha}\to (gh)h^{-\alpha-1}$$
is an isomorphism of filtered $\Dmod$-modules. We use this to put a structure of filtered
$\Dmod$-module on $\Mmod(h^{\beta})$ for every $\beta\in\QQ$, such that for every $r\in\ZZ$,
we have an isomorphism of filtered $\Dmod$-modules
$$\Mmod(h^{\beta})\to\Mmod(h^{\beta-r}),\quad gh^{\beta}\to (gh^r)h^{\beta-r}.$$
For example, we have have an isomorphism of filtered $\Dmod$-modules $\Mmod(h^0)\simeq \shO_X(*Z)$.

\begin{remark}\label{rem_indep_equations1}
Suppose that $h, \bar h \in\shO_X(X)$ are nonzero, and $\alpha,\bar{\alpha}\in\QQ_{>0}$ are 
such that we have the equality of $\QQ$-divisors
$$\alpha \cdot {\rm div}(h)=\bar{\alpha}\cdot {\rm div}(\bar{h}).$$
Let $\ell$ be  a positive integer such that $\ell\alpha,\ell\bar{\alpha}\in\ZZ$.
In this case there is $g\in\shO_X^*(X)$ such that $h^{\ell\alpha}=g{\bar{h}}^{\ell\bar{\alpha}}$.
 Suppose now that there exists $G\in\shO_X(X)$ such that $G^{\ell}=g$. (For example, this holds after pulling-back to the
\'{e}tale cover $\Spec \shO_X[z]/(z^{\ell}-g)$.)
In this case we have an isomorphism of filtered $\Dmod_X$-modules
$$\Phi\colon \Mmod(h^{-\alpha})\longrightarrow \Mmod(\bar{h}^{-\bar{\alpha}})$$
given by
$$\Phi(wh^{-\alpha})=wG^{-1}\bar{h}^{-\bar{\alpha}}.$$
Indeed, this follows from the definition of the filtrations and the isomorphism of schemes over $U$
$$\varphi\colon {\bf Spec} \shO_U[y]/(y^{\ell}-\bar{h}^{-\ell\bar{\alpha}})\to {\bf Spec} \shO_U[y]/(y^{\ell}-h^{-\ell\alpha})$$
that pulls-back $y$ to $G^{-1}y$. 
\end{remark}

\begin{remark}\label{rem_behavior_etale}
It is clear that the filtration on $\Mmod(h^{-\alpha})$ is compatible with restriction to open subsets.
More generally, it is compatible with smooth pullback, as follows.
Suppose that $h\in\shO_X(X)$ is nonzero and $\alpha \in \QQ$.
If $\varphi\colon Y\to X$ is a smooth morphism and $g=h\circ \varphi$, then there is an isomorphism of $\Dmod_Y$-modules
$$\Mmod(g^{-\alpha})\simeq\varphi^*\Mmod(h^{-\alpha}),$$
such that for every $k$ we have
$$F_k\Mmod(g^{-\alpha})\simeq\varphi^*F_k\Mmod(h^{-\alpha}).$$
Indeed, choose $\ell\geq 2$ such that $\ell\alpha\in\ZZ$ and consider the 
Cartesian diagram
$$
\begin{tikzcd}
V_Y \rar{p_Y}\dar{\psi} & U_Y\rar{j_Y}\dar & Y\dar{\varphi} \\
V\rar{p} & U\rar{j} & X,
\end{tikzcd}
$$
where $j$ and $p$ are as in Lemma~\ref{lem_decomp1} and $j_Y$ and $p_Y$ are the corresponding morphisms for $Y$ and $g$. 
Note that we have a base-change theorem that gives
\begin{equation}\label{eq_rem_behavior_etale}
\varphi^!j_+p_+\QQ^H_V[n]\simeq (j_Y)_+(p_Y)_+\psi^!\QQ^H_V[n]
\end{equation}
(see \cite[(4.4.3)]{Saito-MHM}).
Moreover, since $\varphi$ is smooth, if $d=\dim(Y)-\dim(X)$, then for every filtered $\Dmod$-module 
$(\Mmod,F)$ underlying a mixed Hodge module $M$, the filtered $\Dmod$-module
underlying $\varphi^!M$ is $(\varphi^*\Mmod,F)[d]$, where $F_k(\varphi^*\Mmod)=\varphi^*(F_{k}\Mmod)$
(see \cite[3.5]{Saito-MHP}).
This also applies to $\psi$; in particular, we have $\psi^!\QQ_V^H[n]\simeq\QQ_{V_Y}^H[n+2d]$.
By decomposing both sides of (\ref{eq_rem_behavior_etale}) with respect to the $\mu_{\ell}$-action, we
obtain our assertion.
\end{remark}

\subsection{The case of smooth divisors}\label{smooth}

Our goal now is to describe the filtrations introduced in the previous section when $Z$ is a smooth divisor. We will then use this to define
Hodge ideals for arbitrary $\QQ$-divisors. The key result in the smooth case is the following:

\begin{lemma}\label{lem_decomp2}
Let 
 $$\psi\colon Y=\Spec \CC[t]\longrightarrow X=\Spec \CC[x]$$ be
  the map given by $\psi^*(x)=t^{\ell}$. If $Z$ is the divisor on $Y$ defined by
$t$, then
we have
an isomorphism of filtered $\Dmod_X$-modules
$$\psi_+\shO_{Y}(*Z)\simeq\bigoplus_{j=0}^{\ell-1}\Mmod_j,$$
where 
$\Mmod_j\simeq\Dmod_X/\Dmod_X(\partial_x x - \frac{j}{\ell})$ 
and $F_k\Mmod_j$ is generated over $\shO_X$ by the classes of $1,\partial_x,\ldots,\partial_x^k$.
Moreover, if we consider on $Y$ the $\mu_{\ell}$-action such that every $\lambda\in\mu_{\ell}$ maps $t$ to $\lambda t$, then
$\Mmod_j$ is
the component of $\psi_+\shO_{Y}(*Z)$
on which every  $\lambda\in\mu_{\ell}$ acts by multiplication with $\lambda^j$.
\end{lemma}

\begin{proof}
As usual, it is easier to do the computation for the filtered right $\Dmod$-module $\omega_Y(*Z)$ corresponding to $\shO_Y(*Z)$. 
Note that this is filtered quasi-isomorphic to the complex
$$A^{\bullet}:\,\,0\longrightarrow \Dmod_Y\overset{w}\longrightarrow \omega_Y(Z)\otimes_{\shO_Y}\Dmod_Y\longrightarrow 0,$$
placed in degrees $-1$ and $0$, where $w(1)=\frac{dt}{t}\otimes t\partial_t$; see e.g. \cite[Proposition 3.1]{MP1}. 
Since $\psi$ is finite, the functor $\psi_*$ is exact on quasi-coherent $\shO_Y$-modules, hence
$\psi_+\omega_Y(*Z)$ is computed by the $0$-th cohomology of the complex
$$B^{\bullet}=\psi_*(A^{\bullet}\otimes_{\Dmod_Y}\Dmod_{Y\to X}),$$
with the obvious induced filtration. The definition of $w$
immediately implies that $w\otimes 1_{\Dmod_{Y\to X}}$ is injective. 
Note that $\frac{dt}{t}=\frac{1}{\ell}\frac{dx}{x}$ and $t\partial_t=\ell x\partial_x$.

In order to describe the complex $B^{\bullet}$, note that any element of $B^{-1}$ can be uniquely written as $\sum_{j=0}^{\ell-1}t^jP_j$,
with $P_j\in \Dmod_X$. Similarly, any element in $B^0$ can be uniquely written as $\sum_{j=0}^{\ell-1}t^j\frac{dx}{x}Q_j$, with $Q_j\in\Dmod_X$.
Moreover, if $\tau$ is the differential in $B^{\bullet}$, then
$$\tau\left(\sum_{j=0}^{\ell-1}t^jP_j\right)=\sum_{j=0}^{\ell-1}t^j\frac{dx}{x}(x\partial_x+\frac{j}{\ell})P_j,$$
where we use the fact that $t\partial_tt^j=t^jt\partial_t+jt^j$. In other words, we have have an eigenspace decomposition 
$$B^{\bullet}\simeq \bigoplus_{j=0}^{\ell-1}B^{\bullet}_j,$$
where $B^{\bullet}_j$ is identified with the complex
$$0\to\Dmod_X\to \Dmod_X\to 0,$$
with the differential mapping $P$ to $(x\partial_x+\frac{j}{\ell})P$.
It follows that $B^{\bullet}$ is filtered quasi-isomorphic to
$$\bigoplus_{j=0}^{\ell-1}\Dmod_X/(x\partial_x+\frac{j}{\ell})\Dmod_X,$$
where the filtration on the $j$-th component is such that
$$F_{k-1}\left(\Dmod_X/(x\partial_x+\frac{j}{\ell})\Dmod_X\right)$$
is the $\shO_X$-submodule generated by the classes of $1,\partial_x,\ldots,\partial_x^k$.
Moreover, every $\lambda\in\mu_{\ell}$ acts on the $j^{\rm th}$ factor in the above decomposition
by multiplication with $\lambda^j$. 

The assertion in the lemma now follows immediately from the explicit description
of the equivalence between the categories of left and right $\Dmod$-modules on $X={\mathbf A}^1$. 
Indeed, recall that if $\tau$ is the $\CC$-linear endomorphism of the Weyl algebra 
$\Gamma({\mathbf A}^1,\Dmod_{{\mathbf A}^1})$
such that $\tau(PQ)=\tau(Q)\cdot\tau(P)$ for all $P$ and $Q$, and such that $\tau(t)=t$ and $\tau(\partial_t)=-\partial_t$, then the left
$\Dmod$-module $N$ corresponding to a right $\Dmod$-module $M$ is isomorphic to $M$ itself, with scalar multiplication given via the map $\tau$.
Moreover, for filtered $\Dmod$-modules, via this isomorphism $F_kN$ corresponds to $F_{k-1}M$.
In particular, we see that if $M=\Dmod_X/P\cdot\Dmod_X$, then $N\simeq \Dmod_X/\Dmod_X\cdot \tau(P)$, and we obtain the statement.
\end{proof}

In what follows, we denote by $\lceil\alpha\rceil$ the smallest integer that is $\geq\alpha$. For a $\QQ$-divisor $D=\sum_{i=1}^ra_iD_i$,
we put $\lceil D\rceil=\sum_{i=1}^r\lceil a_i\rceil D_i$. 

\begin{corollary}\label{smooth_case}
If $h\in \shO_X(X)$ is nonzero and such that the support $Z$ of ${\rm div}(h)$ is smooth (possibly disconnected), then
for every $\alpha\in\QQ_{> 0}$ the filtration on 
$\Mmod(h^{- \alpha})$ is given by
$$F_k\Mmod(h^{-\alpha})=\shO_X\big((k+1)Z-\lceil D\rceil\big)h^{-\alpha}\quad \text{if}\quad k\geq 0,$$
where $D=\alpha\cdot{\rm div}(h)$, 
and $F_k\Mmod(h^{-\alpha})=0$ if $k<0$.
\end{corollary}

\begin{proof}
We first reduce to the case when $Z={\rm div}(h)$. We can check the assertion in the proposition locally, hence we may assume that
$Z={\rm div}(g)$, for some $g\in\shO_X(X)$, and $h=ug^m$, for some $u\in\shO_X^*(X)$. Furthermore, by Remark~\ref{rem_behavior_etale},
it is enough to prove the assertion after passing to a surjective \'{e}tale cover, hence we may assume that $u=v^m$ for some $v\in\shO_X^*(X)$.
After replacing $g$ by $vg$, we may thus assume that $h=g^m$. In this case we have an isomorphism
of filtered $\Dmod$-modules $\Mmod(h^{-\alpha})\simeq\Mmod(g^{-m\alpha})$, hence we may and will assume that
${\rm div}(h)=Z$.

We consider the smallest positive integer $\ell$ such that $m:=\ell\alpha\in\ZZ$.
If $\ell=1$, then the assertion follows from the formula for the filtration on $\shO_X(*Z)$ when $Z$ is smooth; see \cite[Proposition 8.2]{MP1}. Therefore from now on we assume $\ell>1$.

The morphism $h\colon X\to\AAA^1$ is smooth over some open neighborhood of $0$. Using Remark~\ref{rem_behavior_etale}, we see that in order to prove the corollary,
we may assume that $X=\AAA^1$ and $h=x$, the standard coordinate on $\AAA^1$.
Consider the Cartesian diagram 
$$
\begin{tikzcd}
V \rar{j_0}\dar{p} & W\dar{g} \\
U\rar{j} & X,
\end{tikzcd}
$$
where 
$$j_0\colon V=\Spec \CC[x,x^{-1},y]/(y^{\ell}-x^{-m})\to W=\Spec \CC[x,z]/(z^{\ell}-x^m),\quad j_0^*(z)=y^{-1}.$$
Let
 $\varphi\colon\widetilde{W}= \Spec \CC[t] \to W$ be the normalization, given by 
$$\varphi^*(x)=t^{\ell} \,\,\,\, {\rm and} \,\,\,\, \varphi^*(z)=t^m.$$ 
(Here we use that $\ell$ and $m$ are relatively prime.) 
Note that $\varphi$ is an isomorphism over $V$, hence we have an open embedding $\iota\colon V\hookrightarrow \widetilde{W}$, with complement the smooth divisor $T$ defined by $t$
(in fact, if $a$ and $b$ are integers such that $am+b\ell=1$, then $\iota^*(t)=y^{-a}x^b$). 
We thus have 
$$j_+p_+\shO_V\simeq \psi_+\iota_+\shO_V\simeq\psi_+\shO_{\widetilde{W}}(*T),$$
where $\psi=g\circ\varphi$. We apply Lemma~\ref{lem_decomp2} for $\psi$. 
Note that $\varphi$ is a $\mu_{\ell}$-equivariant morphism if we let each $\lambda\in\mu_{\ell}$ act on $t$ by multiplication with $\lambda^a$.
By considering the behavior with respect to the
$\mu_{\ell}$-action, we see that
 in the decomposition given by the lemma, we have
$\Mmod_j\simeq\Mmod(x^{-\alpha})$ if and only if $ja\equiv -1$ (mod $\ell$), that is,
$j\equiv -m$ (mod $\ell$). 

Suppose first that $\alpha<1$, in which case the condition for $j$ is that $j=\ell-m$. As a reality check, note that we indeed have an isomorphism
$$\Dmod_X/\Dmod_X(\partial_xx -\frac{\ell-m}{\ell})\simeq \Mmod(x^{-\alpha})$$
that maps the class of $1$ to $x^{-\alpha}$. 
The formula for the filtration on $\Mmod(h^{\alpha})$ now follows from Lemma~\ref{lem_decomp2}.
When $\alpha>1$, we put $m=\lceil\alpha\rceil-1$, and use the fact from Remark \ref{twist_positive}, namely that we have an isomorphism of filtered modules
$$\Mmod(x^{-\alpha})\to\Mmod(x^{-\alpha+m}),\quad gx^{-\alpha}\to (gx^{-m})x^{-\alpha+m},$$
to reduce the assertion to the case $\alpha\in (0,1)$.
This completes the proof of the corollary. 
\end{proof}

\subsection{Definition of Hodge ideals for $\QQ$-divisors}\label{definition}
In general, we obtain an upper bound for the terms in the filtration on $\Mmod(h^{-\alpha})$
by restricting to the open subset where the support of ${\rm div}(h)$ is smooth, as follows.

\begin{proposition}\label{prop_for_def}
Given a nonzero $h\in\shO_X(X)$ and a positive rational number $\alpha$, 
for every $k\geq 0$ we have 
$$F_k\Mmod(h^{-\alpha})\subseteq \shO_X\big((k+1)Z-\lceil D\rceil \big)h^{-\alpha},$$
where $D=\alpha\cdot {\rm div}(h)$ and $Z={\rm Supp}(D)$,
while $F_k\Mmod(h^{-\alpha})=0$ for $k<0$.
\end{proposition}

\begin{proof}
Let $\iota\colon X_0\to X$ be an open immersion such that the codimension of its image in $X$ is $\geq 2$ and $Z\vert_{X_0}$ is smooth (though possibly disconnected).
Note that our constructions are compatible with restrictions to open subsets. Moreover, since $\Mmod(h^{-\alpha})$ is clearly torsion-free, it follows that
$F_k:=F_k\Mmod(h^{-\alpha})$ is torsion free, hence the canonical map $F_k\to \iota_*\big(F_k\vert_{X_0}\big)$ is injective. Therefore it is enough to prove the assertion on 
$X_0$, hence we may assume that $Z$ is smooth. However, in this case the assertion follows 
from Corollary~\ref{smooth_case}.
\end{proof}

We can now  define the Hodge ideals for $\QQ$-divisors. Let $X$ be a smooth complex algebraic variety and $Z$ a reduced effective divisor on $X$.
Given an effective $\QQ$-divisor $D$ with ${\rm Supp}(D)=Z$, we define coherent ideals sheaves $I_k(D)$  in $\shO_X$ as follows.
Suppose first that there is a nonzero $h\in\shO_X(X)$, with $H = {\rm div}(h)$, and a positive rational number $\alpha$ such that $D=\alpha H$. 
It turns out to be more convenient to work with the $\Dmod_X$-module $\Mmod (h^\beta)$, where $\beta = 1 - \alpha$. Recall that we have a 
filtered isomorphism 
$$\Mmod (h^{-\alpha}) \rightarrow \Mmod (h^\beta), \,\,\,\,\,\,wh^{- \alpha} \to (wh^{-1})h^\beta,$$
and therefore, if $k\geq 0$,
it follows from Proposition~\ref{prop_for_def} that there is a unique coherent ideal $I_k(D)$ such that 
$$F_k\Mmod(h^{\beta})=I_k(D)\otimes_{\shO_X} \shO_X\big(kZ + H \big)h^{\beta}$$
(note that we always have $\lceil D \rceil \ge Z$).
The definition is independent of the choice of $\alpha$ and $h$: indeed, using Remark~\ref{rem_behavior_etale}, it is enough to check this after the pullback by a suitable \'{e}tale surjective map, hence we deduce the independence assertion using Remark~\ref{rem_indep_equations1}. 
This implies that the general case of the definition follows by covering $X$ with suitable affine open subsets such that $D$ can be written as above in each of them. Note that when $D=Z$ we have $\beta = 0$, and so the ideals $I_k(D)$ are the Hodge ideals studied in \cite{MP1}.

\begin{remark}\label{rmk_inclusion_ideals}
From the definition and the filtration property,  it follows that we always have the inclusion 
$$\shO_X (-Z) \cdot I_{k-1}(D) \subseteq I_k (D)\quad\text{for}\quad k\geq 1.$$ 
We note that for the reduced divisor $Z$,
we have the more subtle inclusions
$$I_k (Z) \subseteq I_{k-1} (Z)\quad \text{for}\quad k\geq 1$$
(see \cite[Proposition~13.1]{MP1}). We do not know however whether this holds for arbitrary $\QQ$-divisors $D$, 
and in fact we suspect that this is not the case. (Note that it does hold when $D$ has simple normal crossings support
by Proposition \ref{Hodge_ideals_SNC}. It is also shown to hold when $D$ has an isolated weighted homogeneous singularity 
in the upcoming \cite{Zhang}.)
However, when $D = \alpha Z$ these inclusions do hold modulo the ideal $\shO_X(-Z)$, see \cite[Corollary~B]{MP3}. 
More precisely, we have 
$$I_k(\alpha Z)+\shO_X(-Z)\subseteq I_{k-1}(\alpha Z)+\shO_X(-Z)\quad\text{for}\quad k\geq 1.$$
This implies in particular that if $I_k(\alpha Z)=\shO_X$ for some $k\geq 1$, then $I_{k-1}(\alpha Z)=\shO_X$.
\end{remark}

\begin{remark}\label{old_ideal}
According to Proposition \ref{prop_for_def}, we also have ideals $I_k^\prime (D)$ given by
$$F_k\Mmod(h^{-\alpha})=I_k^\prime (D)\otimes_{\shO_X} \shO_X\big((k+1)Z -\lceil D\rceil \big)h^{-\alpha},$$
which are related to $I_k(D)$ by the formula
$$I_k (D) = I_k^\prime (D) \otimes_{\shO_X} \shO_X (Z -\lceil D\rceil).$$
\end{remark}

The following periodicity property often allows us to reduce our study to the case  $\lceil D\rceil = Z$.

\begin{lemma}\label{periodicity}
If $D'$ is an integral divisor with ${\rm Supp}(D')\subseteq {\rm Supp}(D)$, then
$$I_k(D+D')=I_k(D)\otimes_{\shO_X} \shO_X (- D').$$
In particular 
$$I_k (D) = I_k (B) \otimes \shO_X (Z - \lceil D \rceil),$$
with $B = D + Z - \lceil D \rceil$ satisfying $\lceil B \rceil = Z$.
\end{lemma}
\begin{proof}
Using the notation in Remark \ref{old_ideal}, the equivalent statement
$$I_k^\prime (D+D')=I_k^\prime (D)$$
follows from the definition and Remark~\ref{rem_isom_comp_with_filtrations}.
\end{proof}

\begin{remark}\label{always_nontrivial}
Note that $I_k (D) \subseteq \shO_X (Z - \lceil D \rceil)$ for all $k$, and so if $\lceil D \rceil \neq Z$, then one can never have 
$I_k (D) = \shO_X$. It is however still interesting to ask whether $I_k (B) = \shO_X$.
\end{remark}

\subsection{A global setting for the study of Hodge ideals}\label{global_setting}
We now consider a setting in which we can define global filtered $\Dmod_X$-modules that are locally isomorphic to the
$\big(\Mmod(h^{-\alpha}), F\big)$ discussed in the previous sections. 

Let $X$ be a smooth variety and $D=\frac{1}{\ell}H$ a $\QQ$-divisor, where $H$ is an integral divisor and $\ell$ is a positive integer.
The extra assumption we make here is that there is a line bundle $M$ such that 
$$\shO_X(H)\simeq M^{\otimes \ell}.$$ 
We denote by $U$ the complement of $Z={\rm Supp}(H)$
and by $j$ the inclusion $U\hookrightarrow X$.

Let $s\in \Gamma(X,M^{\otimes \ell})$ be a section whose zero locus is $H$. 
Since $s$ does not vanish on $U$, we may consider the section $s^{-1}\in\Gamma\big(U,(M^{-1})^{\otimes \ell}\big)$. Let $p\colon V\to U$ be the \'{e}tale cyclic cover corresponding 
to $s^{-1}$, hence
$$V\simeq {\bf Spec}\big(\shO_U\oplus M\oplus\ldots\oplus M^{\otimes(\ell-1)}\big).$$
We consider the filtered $\Dmod_X$-module $\Mmod=j_+p_+\shO_V$, that underlies
a mixed Hodge module. The  $\mu_{\ell}$-action on $V$, where $\lambda\in\mu_{\ell}$ acts on $M^{\otimes i}$ by multiplication with
$\lambda^{-i}$,
induces an eigenspace decomposition
$$\Mmod=\bigoplus_{i=0}^{\ell-1}\Mmod_i,$$
where $\lambda\in\mu_{\ell}$ acts on $\Mmod_i$ by multiplication with $\lambda^{-i}$.
We consider on each $\Mmod_i$ the induced filtration.

Note that if $X_0$ is an open subset of $X$ such that we have a trivialization $M\vert_{X_0}\simeq\shO_{X_0}$, and if 
via the corresponding trivialization of $M^{\otimes \ell}\vert_{X_0}$, the restriction $s\vert_{X_0}$ corresponds to $h_0\in\shO_X(X_0)$, then 
we have isomorphisms of filtered $\Dmod_{X_0}$-modules
$$\Mmod_i\simeq\Mmod(h_0^{-i/\ell})\quad\text{for}\quad 0\leq i\leq \ell-1.$$
We also see that the filtration on $\Mmod$ is the direct sum filtration, since this holds locally. 
Moreover, we have isomorphisms of $\shO_{X_0}$-modules
$$\Mmod_i\vert_{X_0}\simeq\shO_X(*Z)\vert_{X_0},$$
which glue to isomorphisms of $\shO_X$-modules
$$\Mmod_i\simeq M^{\otimes i}\otimes_{\shO_X}\shO_X(*Z)=j_*j^*M^{\otimes i}.$$
Via these isomorphisms, it follows from the definition of Hodge ideals (see also Remark \ref{old_ideal}) that we have
$$F_k\Mmod_i\simeq M^{\otimes i}\otimes_{\shO_X}I_k^\prime \left(i/\ell \cdot H\right)\otimes_{\shO_X}\shO_X\left((k+1)Z-\lceil i/\ell \cdot H\rceil\right)
\simeq$$ 
$$\simeq M^{\otimes i} (-H) \otimes_{\shO_X}I_k\left(i/\ell \cdot H\right)\otimes_{\shO_X}\shO_X\left(kZ+ H\right).$$

\subsection{A complex associated to simple normal crossing divisors}\label{complex_for_SNC}
We now discuss a complex that, as we will see later, gives
a filtered resolution of $\Mmod_r(h^{-\alpha})$ by filtered induced $\Dmod_X$-modules in the case when $h$ defines a simple normal crossing divisor.

Let $X$ be a smooth, $n$-dimensional, complex variety, $h\in\shO_X(X)$ nonzero, and $\alpha$ a nonzero rational number (we allow $\alpha$ 
to be either positive or negative). Let $D = \alpha \cdot {\rm div}(h)$. We denote by $Z$
the support of $D$, and assume that it has simple normal crossings.

Associated to $Z$ we have the following complex of right $\Dmod_X$-modules:
$$C^{\bullet}:\,\,0\to \Dmod_X\to\Omega_X^1(\log Z)\otimes_{\shO_X}\Dmod_X\to\cdots\to \Omega_X^n(\log Z)\otimes_{\shO_X}\Dmod_X\to 0,$$
placed in degrees $-n,\ldots,0$. We denote by $D_i\colon C^i\to C^{i+1}$ its differentials. 
If $x_1,\ldots,x_n$ are local coordinates on $X$, then
$$D_i(\eta\otimes P)=d\eta\otimes P+\sum_{i=1}^n (dx_i\wedge\eta)\otimes \partial_{x_i}P.$$
In fact $C^{\bullet}$ is a filtered complex, where
$$F_{p-n}C^i=\Omega_X^{i+n}(\log Z)\otimes_{\shO_X}F_{p+i}\Dmod_X.$$
This filtered complex is quasi-isomorphic to the filtered right $\Dmod_X$-module $\omega_X(*Z)$ 
corresponding to the filtered left $\Dmod_X$-module $\shO_X(*Z)$ (see \cite[Proposition~3.1]{MP1}, and 
\cite[Proposition~3.11(ii)]{Saito-MHM} for a more general statement).

Given $h$ and $\alpha$ as above, we also consider the filtered complex $C^{\bullet}_{h^{-\alpha}}$ consisting of the same sheaves, but 
with differential $C^i_{h^{-\alpha}}\to C^{i+1}_{h^{-\alpha}}$ given by
$$D_i-\big((\alpha\cdot {\rm dlog}(h)\wedge \bullet)\otimes {1_{\Dmod_X}}\big).\footnote
{In related settings, for instance involving the de Rham complex of $\Mmod(h^{-\alpha})$, this type of complex can already be found in the literature; see for instance \cite[\S6.3.11]{Bjork}.}$$
It is easy to see that this is indeed a filtered complex. 

Suppose now that we also have an effective divisor $T$ supported on $Z$. It is not hard to check that the formula for the map
$$C^i_{h^{-\alpha}}\to C^{i+1}_{h^{-\alpha}}$$
induces also a map
$$C^i_{h^{-\alpha}}(-T):=\shO_X(-T)\otimes_{\shO_X}\Omega^{i+n}_X(\log Z)\otimes_{\shO_X}\Dmod_X$$
$$\to 
C^{i+1}_{h^{-\alpha}}(-T):=\shO_X(-T)\otimes_{\shO_X}\Omega^{i+1+n}_X(\log Z)\otimes_{\shO_X}\Dmod_X.$$
This is due to the fact that if locally $T={\rm div}(u)$ and $\eta$ is a local section of $\Omega^{i+n}_X(\log Z)$,
then we can write 
$d(u\eta)=ud(\eta)+u\cdot {\rm dlog}(u)\wedge\eta$.
We thus obtain a filtered subcomplex $C^{\bullet}_{h^{-\alpha}}(-T)$ of $C^{\bullet}_{h^{-\alpha}}$.
We emphasize that this is \emph{not} obtained by tensoring $C^{\bullet}_{h^{-\alpha}}$ with $\shO_X(-T)$.

\begin{proposition}\label{prop_SNC_complex}
If no coefficient of $D - T$ lies in $\ZZ_{<0}$, then
the complex $C^{\bullet}_{h^{-\alpha}}(-T)$
is filtered quasi-isomorphic to $\big(h^{-\alpha}\omega_X(*Z), G_\bullet \big)$, where 
$$G_{k-n}h^{-\alpha}\omega_X(*Z)=0\quad\text{if}\quad k<0,$$
$$G_{-n}h^{-\alpha}\omega_X(*Z)=h^{-\alpha}\omega_X(Z- T)\quad\text{and}$$
$$G_{k-n}h^{-\alpha}\omega_X(*Z)=G_{-n}h^{-\alpha}\omega_X(*Z)\cdot F_k\Dmod_X\quad\text{if}\quad k>0.$$
\end{proposition}

\begin{proof}
It is immediate to check that the differential induced on ${\rm gr}^F_pC^{\bullet}_{h^{-\alpha}}(-T)$ does become 
equal to the differential $D_i$ twisted with the identity on $\shO_X(-T)$, and therefore for every $p$ we have
$${\rm gr}^F_pC^{\bullet}_{h^{-\alpha}}(-T)=\shO_X(-T)\otimes_{\shO_X}{\rm gr}^F_pC^{\bullet}.$$
In particular, we have 
$$H^iF_pC^{\bullet}_{h^{-\alpha}}(-T)=0\quad\text{for every}\quad p\in\ZZ\,\,\text{and}\,\, i\in\ZZ\smallsetminus\{0\},$$
by the result in \cite{MP1} quoted above. Consider now the morphism of right $\Dmod_X$-modules
$$\varphi\colon C^0_{h^{-\alpha}}(-T) = \omega_X (Z - T)\otimes_{\shO_X} \Dmod_X \longrightarrow 
h^{-\alpha}\omega_X(*Z)$$
given by 
$$\varphi(w\otimes\eta\otimes Q)=(h^{-\alpha}w\eta)Q.$$
We first check that this morphism is surjective. We do this locally, hence we may assume that we have a system of coordinates $x_1,\ldots,x_n$ on $X$ such that $\shO_X(-Z)$ is generated by $x_1\cdots x_r$ and $\shO_X(-T)$ by $x_1^{\beta_1}\cdots
x_r^{\beta_r}$. We also write $h=u x_1^{a_1}\cdots x_r^{a_r},$ where $u$ is an everywhere nonvanishing function, 
and define $\alpha_i=\alpha a_i$ and $\gamma_i=\alpha_i-\beta_i$ for all $i$. Note for later use that
$$\alpha\cdot {\rm dlog}(h)=\frac{\alpha}{u}du+\sum_{i=1}^r\alpha_i\frac{dx_i}{x_i}.$$

The surjectivity of $\varphi$ follows from the fact that 
\begin{equation}\label{eq0_prop_SNC_complex}
{\rm Im}(\varphi)=(h^{-\alpha}x_1^{\beta_1}\cdots x_r^{\beta_r}\eta)\cdot\Dmod_X=h^{-\alpha}\omega_X(*Z),
\end{equation}
where $$\eta={\rm dlog}(x_1)\wedge\ldots\wedge
{\rm dlog}(x_r)\wedge dx_{r+1}\wedge\ldots\wedge dx_n,$$ and the second equality in (\ref{eq0_prop_SNC_complex})
   is a consequence of the fact that $-\gamma_i-1\not\in \ZZ_{\geq 0}$ for all $i$, by assumption.

In order to complete the proof of the proposition it is enough to show
that, for every $k\geq 0$, the following sequence is exact:
$$
\shO_X(-T)\otimes\Omega_X^{n-1}(\log Z)\otimes F_{k-1}\Dmod_X\overset{\psi_k}\longrightarrow\shO_X(-T)\otimes \omega_X(Z)\otimes F_k\Dmod_X\overset{\varphi_k}\longrightarrow$$ 
$$\overset{\varphi_k}\longrightarrow G_{k-n} h^{-\alpha}\omega_X(*Z) \longrightarrow 0
$$
where $\varphi_k$ is the restriction of $\varphi$ to the $(k-n)$-th level of the filtration and $\psi_k$ is the restriction of the  differential of  $C^{\bullet}_{h^{-\alpha}}(-T)$.
The surjectivity of $\varphi_k$ is an immediate consequence of the surjectivity of $\varphi$ and the definition of the filtration on $h^{\alpha}\omega_X(*Z)$.

Keeping the above notation for the local coordinates on $X$,
it follows from the definition of $\psi_k$ that
$${\rm Im}(\psi_k)=\prod_{j=1}^rx_j^{\beta_j}\otimes \eta \otimes \big(\sum_{i=1}^r\big(x_i\partial_i-\gamma_i-\frac{\partial u}{\partial x_i}\cdot \frac{\alpha x_i}{u}\big)\cdot F_{k-1}\Dmod_X + $$
$$ + \sum_{i=r+1}^n
\big(\partial_i-\frac{\partial u}{\partial x_i}\cdot\frac{\alpha}{u}\big)\cdot F_{k-1}\Dmod_X\big)$$
and it is straightforward to see that this is contained in ${\rm Ker}(\varphi_k)$. We now prove by induction on $k$
that if $\varphi_k(x_1^{\beta_1}\cdots x_r^{\beta_r}\otimes\eta\otimes P)=0$ for some $P\in F_k\Dmod_X$, then $x_1^{\beta_1}\cdots x_r^{\beta_r}\otimes\eta \otimes P\in {\rm Im}(\psi_k)$. 
Note that the case $k=0$ is trivial.
Let's write $P=\sum_{u,v}c_{u,v}\partial^{u}x^{v}$, where $u$ and $v$ vary over $\ZZ_{\geq 0}^n$. 
After subtracting suitable terms from $P$, we may assume that whenever $c_{u,v}\neq 0$, we have $u_i=0$ for $i>r$.
Furthermore, note that if $u_i,v_i>0$ for some $i\leq r$, then we can write
$$\partial^{u}x^{v}=(x_i\partial_i-\gamma_i-\frac{\partial u}{\partial x_i}\cdot \frac{\alpha x_i}{u})A+B,$$
with both $A$ and $B$ of order $\leq k-1$. Therefore we may also assume that whenever $c_{u,v}\neq 0$ and $|u|:=\sum_iu_i=k$, we have
\begin{equation}\label{eq_prop_SNC_complex}
u_iv_i=0\quad\text{for}\quad 1\leq i\leq n.
\end{equation}
%We order $\ZZ^n$ lexicographically and argue by induction on the top degree $\overline{\beta}$ among those $\beta$ with $c_{{\beta},\gamma}\neq 0$ for some $\gamma$. 
%Let $s\leq r$ be such that $\overline{\beta}_{s-1}=0$ and $\overline{\beta_s}\neq 0$ (note that by assumption, in this case $\beta_i=0$ for all $i\leq s-1$
%whenever $c_{\beta,\gamma}\neq 0$). 
Since 
$$(h^{-\alpha}x_1^{\beta_1}\cdots x_r^{\beta_r}\eta)\partial^{u}x^{v}=(\text{non-zero constant}\footnote{Here we use again the fact that
$-\gamma_i-1\not\in\ZZ_{\geq 0}$ for all $i$.})\cdot (h^{-\alpha}x_1^{\beta_1}\cdots x_r^{\beta_r}\eta)x^{v-u},$$
and since (\ref{eq_prop_SNC_complex}) implies that for every $(u,v)$ and $(u',v')$ with $|u|=k$ and $c_{u,v},c_{u',v'}\neq 0$
we have $x^{v-u}\neq x^{v'-u'}$, we conclude that in fact $P\in F_{k-1}\Dmod_X$, hence we are done by induction.
\end{proof}

%If $\alpha>0$, then $T=\lceil D\rceil$ satisfies the hypothesis of the proposition, and we obtain

%\begin{corollary}\label{special_SNC_complex}
%In the setting above, if $\alpha>0$, then the complex $C^{\bullet}_{h^{-\alpha}}(- \lceil D \rceil)$ is filtered quasi-isomorphic to $\Mmod_r (h^{-\alpha})$.
%\end{corollary}

\subsection{The Hodge ideals of simple normal crossing divisors}\label{SNC}
In this section we show that the Hodge ideals of divisors with simple normal crossing support essentially depend only on the support of the divisor, and therefore can be computed as in \cite[\S8]{MP1}.

\begin{proposition}\label{Hodge_ideals_SNC}
Let $X$ be a smooth variety, and $D$ an effective divisor on $X$ with simple normal crossing  support $Z$. Then for all $k$ we have 
$$I_k(D)=I_k(Z)\otimes_{\shO_X} \shO_X(Z - \lceil D \rceil).$$
\end{proposition}

\begin{proof}
Equivalently, we need to show that $I_k^\prime (D) = I_k(Z)$ for all $k$. The assertion is local, hence we may assume that we have coordinates $x_1,\ldots,x_n$
on $X$ such that $Z=H_1+\cdots+H_r$, where $H_i$ is defined by $x_i=0$. The morphism $X\to \CC^r$ given by 
$(x_1,\ldots,x_r)$ is smooth, hence
using Remark~\ref{rem_behavior_etale} we see that it is enough to prove the proposition when 
$X= \Spec\,\CC[x_1,\ldots,x_n]$ and $D=\sum_{i=1}^n\alpha_iH_i$,
where $H_i= {\rm div}(x_i)$ and $\alpha_i>0$. Let $\ell$ be the smallest positive integer such that all $a_i:=\ell\alpha_i$ are integers. The assertion to be proved is trivial 
when $\ell=1$, hence from now on we assume $\ell\geq 2$. 
Consider the Cartesian diagram 
$$
\begin{tikzcd}
V \rar{j_0}\dar{p} & W\dar{g} \\
U\rar{j} & X,
\end{tikzcd}
$$
where 
$$j_0\colon V=\Spec\,\CC[x_1^{\pm 1},\ldots,x_n^{\pm 1},y]/(y^{\ell}-x_1^{-a_1}\cdots x_n^{-a_n})$$
$$\to W=\Spec\,\CC[x_1,\ldots,x_n,z]/(z^{\ell}-x_1^{a_1}\cdots x_n^{a_n}),$$
with $j_0^*(z)=y^{-1}$. 
We will make use of some standard facts about cyclic covers with respect to simple normal crossing divisors,
exploiting the toric variety structure on the normalization of $W$. For basic facts regarding toric varieties, 
we refer to 
\cite{Fulton}.

Let $N$ be the lattice $\ZZ^n$ and $M$ its dual. We also consider the lattice 
$$N'=\{(v_1,\ldots,v_{n+1})\in\ZZ^{n+1}\mid a_1v_1+\cdots+a_nv_n=\ell v_{n+1}\}$$
and its dual 
$$M'=\ZZ^{n+1}/\ZZ\cdot (a_1,\ldots,a_n,-\ell).$$
Note that we have an injective lattice map $N'\to N$, with finite cokernel, induced by the projection onto the first $n$ components, and the dual map $M\to M'$
is again injective, with finite cokernel. In fact, we have an isomorphism $M'/M\simeq\ZZ/ \ell\ZZ$ that maps the class of $(u_1,\ldots,u_{n+1})\in M'$
to the class of $u_{n+1}$ in $\ZZ/\ell \ZZ$. 

We thus have an isomorphism $N'_{\RR}\simeq N_{\RR}=\RR^n$. The strongly convex cone $\sigma=\RR_{\geq 0}^n$ in 
$N_{\RR}=\RR^n$ gives the toric variety $X=\CC^n$. As a cone in $N'_{\RR}$, 
 $\sigma$ gives an affine toric variety $\widetilde{W}$, and the lattice map $N'\to N$ corresponds to a toric map
$\psi\colon \widetilde{W}\to X$. Note that we have a morphism of $\shO(X)$-algebras 
$\shO(W)\to \shO(\widetilde{W})$ that maps $x_i$ to the element of $\CC[\sigma^{\vee}\cap M']$
corresponding to the class of the $i$-th element of the standard basis of $\ZZ^n$, and $z$ to the class of $(0,\ldots,0,1)$. It is easy to check that
if we denote by $\lfloor \gamma\rfloor$ the largest integer $\leq\gamma$, then
\begin{equation}\label{eq_decomp_W}
\shO(\widetilde{W})=\bigoplus_{0\leq j\leq\ell-1}\shO (X)x_1^{-\lfloor j\alpha_1\rfloor}\cdots x_n^{-\lfloor j\alpha_n\rfloor}z^j,
\end{equation}
and consequently to deduce that $\shO(\widetilde{W})$ is integral over $\shO(W)$. As the coordinate ring of a toric variety, $\shO(\widetilde{W})$
is normal, hence it is the integral closure of $\shO(W)$ in its field of fractions.
Moreover, since $\widetilde{W}$ is a toric variety, we may choose a toric resolution of singularities $Y\to \widetilde{W}$, and let $f\colon Y\to X$ be the composition. Since the map
$Y\to W$ is an isomorphism over the complement of $g^{-1} (\sum H_i)$, it follows that there is an open embedding $\iota\colon V\hookrightarrow Y$ such that $f\circ\iota =j\circ p$. The support $E_Y$ of $Y\smallsetminus
\iota(V)$ is the sum of all prime toric divisors on $Y$. 

The action of the torus $T_{M'}={\rm Spec}\,\CC[M']$ on $\widetilde{W}$ induces an action of the finite group
${\rm Spec}\,\CC[M'/M]\simeq {\rm Spec}\,\CC[\ZZ/\ell\ZZ]=\mu_{\ell}$ on 
$\widetilde{W}$. This is the action induced on the normalization $\widetilde{W}$ by the $\mu_{\ell}$-action on $W$ that we discussed in \S\ref{coverings}. In particular,
the toric resolution $Y\to\widetilde{W}$ is automatically equivariant. Note that in the decomposition (\ref{eq_decomp_W}), an element $\lambda\in\mu_{\ell}$ acts on the 
summand corresponding to $j$ by multiplication with $\lambda^j$.

The equality $f\circ\iota =j\circ p$ implies that we have an isomorphism of filtered $\Dmod_X$-modules
$$j_+p_+\shO_V\simeq f_+\iota_+\shO_V=f_+\shO_Y(*E_Y).$$
As usual, in order to compute the push-forward of $\shO_Y(*E_Y)$, it is more convenient to work with right $\Dmod$-modules.
Recall that there is a complex of right $\Dmod_Y$-modules
$$A^\bullet = A_Y^{\bullet}:\,\,\,0 \rightarrow \Dmod_Y \rightarrow \Omega_Y^1(\log E_Y) \otimes_{\shO_Y} \Dmod_Y \rightarrow \cdots  \rightarrow \omega_Y(E_Y) \otimes_{\shO_Y} \Dmod_Y \rightarrow 0$$
located in degrees $-n,\ldots,0$, that is filtered quasi-isomorphic to $\omega_Y(*E_Y)$; see the beginning of \S\ref{complex_for_SNC}. 
Since $Y$ is a toric variety, we have a canonical isomorphism 
$\Omega_Y^1(\log E_Y)\simeq M'\otimes_{\ZZ}\shO_Y$ (see \cite[Section~4.3]{Fulton}). 
We will also consider the corresponding complex on $X$:
$$A_X^{\bullet}:\,\,\,0\longrightarrow\Dmod_X\longrightarrow M\otimes_{\ZZ}\Dmod_X\longrightarrow \cdots\longrightarrow \wedge^nM\otimes_{\ZZ}\Dmod_X\longrightarrow 0.$$

It follows from the definition that,  forgetting about the filtration, we have
$$f_+\omega_Y(*E_Y)=\derR f_*(A^{\bullet}\otimes_{\Dmod_Y}\Dmod_{Y\to X}).$$
Note that $\Dmod_{Y\to X}=f^* \Dmod_X$ as $\shO_Y$-modules, hence the projection formula implies
$$R^if_*(A^{p-n}\otimes_{\Dmod_Y}\Dmod_{Y\to X})\simeq \wedge^pM'\otimes_{\ZZ}R^if_*\shO_Y\otimes_{\shO_X}\Dmod_X=0$$
for $i>0$, since $f$ is the composition of a finite map with a toric resolution. Therefore 
$f_+\omega_Y(*E_Y)$ is represented by the complex $B^{\bullet}$, where
$$B^{p-n}=\wedge^pM'\otimes_{\ZZ}\psi_*\shO_{\widetilde{W}}\otimes_{\shO_X}\Dmod_X.$$
In order to describe the differential of this complex, it is convenient to use the isomorphism $M_{\QQ}\simeq M'_{\QQ}$
and the decomposition (\ref{eq_decomp_W}). With a little care, it follows from the definitions that if we put
$$B^{p-n}_j=\wedge^pM_{\QQ}\otimes_{\QQ}\shO_X\cdot x_1^{-\lfloor j\alpha_1\rfloor}\cdots x_n^{-\lfloor j\alpha_n\rfloor}z^j\otimes
_{\shO_X}\Dmod_X,$$
then $B^{\bullet}$ decomposes as the direct sum of the subcomplexes $B^{\bullet}_j$, for $0\leq j\leq \ell-1$. 
Furthermore, if we identify each $B^{p-n}_j$ in the obvious way with $A^{p-n}_X$, then the differential 
$$\delta_{B_j}^{p-n}\colon \wedge^pM_{\QQ}\otimes_{\QQ}\Dmod_X\to \wedge^{p+1} M_{\QQ}\otimes_{\QQ}\Dmod_X$$
is given by 
$$\delta_{B_j}^{p-n}=\delta_{A_X}^{p-n}+(w_j \wedge -)\otimes {\rm Id}_{\Dmod_X},$$
where $\delta_{A_X}$ is the differential on $A_X^{\bullet}$ and
$$w_j=(w_{j,1},\ldots,w_{j,n}),\quad\text{with}\quad w_{j,i}=j\alpha_i-\lfloor j\alpha_i\rfloor.$$
It follows from Proposition~\ref{prop_SNC_complex} that we have a morphism
$$B_j^0\to \Mmod_r(x_1^{w_{j,1}}\cdots x_n^{w_{j,n}})$$ 
that induces a quasi-isomorphism
$$B_j^{\bullet}\to \Mmod_r(x_1^{w_{j,1}}\cdots x_n^{w_{j,n}})$$
(see also Remark~\ref{twist_individual}).

We now bring the filtrations into the picture. It follows from Saito's strictness results (see the discussion in 
\S\ref{Hodge_modules}; cf. also \cite[\S4,~\S6]{MP1}) that 
$$F_kf_+\omega_Y(*E_Y)={\rm Im}\big(\derR f_*F_k(A^{\bullet}\otimes_{\Dmod_Y}\Dmod_{Y\to X})\to \derR f_*(A^{\bullet}\otimes_{\Dmod_Y}\Dmod_{Y\to X})\big).$$
Arguing as above, we deduce that 
$$F_kf_+\omega_Y(*E_Y)={\rm Im}\big(f_*F_k(A^{\bullet}\otimes_{\Dmod_Y}\Dmod_{Y\to X})\to f_*(A^{\bullet}\otimes_{\Dmod_Y}\Dmod_{Y\to X})\big).$$
In other words, $(f_+\omega_Y(*E_Y), F)$ is represented by the filtered complex $B^{\bullet}$, and using Proposition~\ref{prop_SNC_complex},
we conclude that 
$$f_+\omega_Y(*E_Y)\simeq \bigoplus_{j=0}^{\ell-1}\Mmod_r(x_1^{w_{j,1}}\cdots x_n^{w_{j,n}}),$$
where the filtration on $\Mmod_r(x_1^{w_{j,1}}\cdots x_n^{w_{j,n}})$ is given by 
$$F_{-n}\Mmod_r(x_1^{w_{j,1}}\cdots x_n^{w_{j,n}})=x_1^{w_{j,1}}\cdots x_n^{w_{j,n}}\omega_X(Z)\quad\text{and}$$
$$F_{k-n}\Mmod_r(x_1^{w_{j,1}}\cdots x_n^{w_{j,n}})=F_{-n}\Mmod_r(x_1^{w_{j,1}}\cdots x_n^{w_{j,n}})\cdot F_k\Dmod_X\quad\text{for}\quad k\geq 1.$$

By comparing the $\mu_{\ell}$-actions, we conclude that the summand 
$\Mmod_r(x_1^{-\alpha_1}\cdots x_n^{-\alpha_n})$ on which an element $\lambda\in\mu_{\ell}$ acts by multiplication with $\lambda^{-1}$
corresponds to $j=\ell-1$. Therefore the filtration on 
$\Mmod(x_1^{-\alpha_1}\cdots x_n^{-\alpha_n})$ is given by
$$F_{-n}\Mmod_r(x_1^{-\alpha_1}\cdots x_n^{-\alpha_n})=(x_1^{-\alpha_1}\cdots x_n^{-\alpha_n})x_1^{\lceil\alpha_1\rceil}\cdots x_n^{\lceil\alpha_n\rceil}\omega_X(Z)\quad\text{and}$$
$$F_{k-n}\Mmod_r(x_1^{-\alpha_1}\cdots x_n^{-\alpha_n})=F_{-n}\Mmod_r(x_1^{-\alpha_1}\cdots x_n^{-\alpha_n})\cdot F_k\Dmod_X\quad\text{for}\quad k\geq 1.$$
It is now a straightforward computation to see that
$I_k^\prime(D)$ is the ideal generated by the monomials $\prod_{i=1}^nx_i^{c_i}$, where $0\leq c_i\leq k$ for all $i$ and $\sum_ic_i=(n-1)k$. 
This coincides with $I_k (Z)$ according to \cite[Proposition~8.2]{MP1},
completing the proof of the proposition.
\end{proof}

\subsection{Computation in terms of a log resolution}\label{log_resolution}
We use the results of the previous two sections in order to describe Hodge ideals of $\QQ$-divisors in terms of log resolutions.
Let $X$ be a smooth variety, $h\in\shO_X(X)$ a nonzero function, $H = {\rm div}(h)$, and $\alpha\in\QQ_{>0}$. We are interested in computing $I_k(D)$, where $D=\alpha H$.
As always, let $Z={\rm Supp}(D)$ and $\beta = 1 - \alpha$.

Let $f\colon Y\to X$ be a log resolution of the pair $(X, D)$ that is an isomorphism over $U=X\smallsetminus Z$,  
and denote $g=h\circ f\in\shO_Y(Y)$. 
We fix a positive integer $\ell$ such that $\ell\alpha\in\ZZ$. 
As usual, we consider 
$$p\colon V={\bf Spec}\,\shO_U[y]/(y^{\ell}-h^{-\ell\alpha})\longrightarrow U$$ 
and the inclusion $j\colon U\hookrightarrow X$. By assumption, we also have an open immersion
$\iota\colon U\hookrightarrow Y$ such that $f\circ\iota=j$. 
 By considering the decompositions of 
$$j_+p_+\shO_V\simeq f_+\iota_+p_+\shO_V$$
into isotypical components, we conclude that we have a filtered isomorphism
\begin{equation}\label{birational}
\Mmod(h^{-\alpha})\simeq f_+\Mmod(g^{-\alpha}).
\end{equation}

We now denote $G=f^*D$, and consider on $Y$ the complex introduced in \S\ref{complex_for_SNC}:
$$C^{\bullet}_{g^{-\alpha}}(-\lceil G\rceil):\,\,0\to\shO_Y(-\lceil G\rceil)\otimes_{\shO_Y} \Dmod_Y\to\shO_Y(-\lceil G\rceil)\otimes_{\shO_Y}\Omega^1_Y(\log E)\otimes_{\shO_Y}\Dmod_Y$$
$$\to\cdots\to\shO_Y(-\lceil G\rceil)\otimes_{\shO_Y}\omega_Y(E)\otimes_{\shO_Y}\Dmod_Y\to 0,$$
where $E = (f^*D)_{\rm red}$.
This is placed in degrees $-n,\ldots,0$, and if $x_1,\ldots,x_n$ are local coordinates on $Y$, then its differential is given by
$$\eta\otimes Q\to d\eta\otimes Q+\sum_{i=1}^n (dx_i\wedge \eta)\otimes\partial_iQ-\big(\alpha\cdot {\rm dlog}(g)\wedge\eta\big)\otimes Q.$$

%Note that if $C_{\bullet}$ is a complex of $\shO_Y$-modules that carries a $\mu_r$-action over $X$, then $R^0f_*C^{\bullet}$ caries a corresponding
%eigenspace decomposition. We denote by $(R^0f_*C^{\bullet})_{\lambda={\rm Id}}$ the eigenspace corresponding to the inclusion 
%$\mu_r\hookrightarrow\CC^*$. 

\begin{theorem}\label{formula_log_resolution}
With the above notation, the following hold:
\begin{enumerate}
\item[i)] For every $p\neq 0$ and every $k\in \ZZ$, we have
$$R^pf_*\big(C^{\bullet}_{g^{-\alpha}}(-\lceil G\rceil)\otimes_{\Dmod_Y}\Dmod_{Y\to X}\big)=0$$
and  
$$R^pf_*F_k\big(C^{\bullet}_{g^{-\alpha}}(-\lceil G\rceil)
\otimes_{\Dmod_Y}\Dmod_{Y\to X}\big)=0.$$
\item[ii)] For every $k\in\ZZ$, the natural inclusion induces an injective map
$$ R^0f_*F_k\big(C^{\bullet}_{g^{-\alpha}}(-\lceil G\rceil)\otimes_{\Dmod_Y}\Dmod_{Y\to X}\big)\hookrightarrow R^0f_*\big(C^{\bullet}_{g^{-\alpha}}(-\lceil G\rceil)\otimes_{\Dmod_Y}\Dmod_{Y\to X}\big).$$
\item[iii)] We have a canonical isomorphism
$$R^0f_*\big(C^{\bullet}_{g^{-\alpha}}(-\lceil G\rceil)\otimes_{\Dmod_Y}\Dmod_{Y\to X}\big)\simeq \Mmod_r(h^{-\alpha})$$ 
that induces for every $k\in\ZZ$ an isomorphism
$$R^0f_*F_{k-n}\big(C^{\bullet}_{g^{-\alpha}}(-\lceil G\rceil)\otimes_{\Dmod_Y}\Dmod_{Y\to X}\big)\simeq h^{-\alpha}\omega_X\big((k+1)Z-\lceil D\rceil\big)\otimes_{\shO_X}I_k^\prime (D)\simeq$$
$$\simeq h^{\beta}\omega_X\big(kZ + H \big)\otimes_{\shO_X}I_k(D).$$
\end{enumerate}
\end{theorem}

\begin{proof}
It follows from Lemma~\ref{decomp_filtration}, and from the definition of its filtration, that $\Mmod_r(g^{-\alpha})$ is 
a direct summand of a right Hodge $\Dmod$-module on $Y$. By Saito's strictness of the filtration of (push-forwards of) such $\Dmod$-modules, it follows that for all $k,p\in\ZZ$ the canonical map
$$R^pf_*F_k\big(\Mmod_r(g^{-\alpha})\overset{\derL}{\otimes}_{\Dmod_Y}\Dmod_{Y\to X}\big)\to R^pf_*\big(\Mmod_r(g^{-\alpha})\overset{\derL}{\otimes}_{\Dmod_Y}\Dmod_{Y\to X}\big)$$
is injective, and its image is equal to
$$F_kR^pf_*\big(\Mmod_r(g^{-\alpha})\overset{\derL}{\otimes}_{\Dmod_Y}\Dmod_{Y\to X}\big)$$
(see the discussion in \S\ref{Hodge_modules}).

On the other hand, note that if write 
$G=\alpha\cdot {\rm div}(g)=\sum_i\alpha_iE_i$, then $-\lceil\alpha_i\rceil+\alpha_i\not\in\ZZ_{<0}$ for all $i$.
We may thus apply Proposition~\ref{prop_SNC_complex} for the divisor $G$, with $T=\lceil G\rceil$.
Using Proposition~\ref{Hodge_ideals_SNC} as well,
we see that $C^{\bullet}_{g^{-\alpha}}(-\lceil G\rceil)$ is filtered quasi-isomorphic to
$\Mmod_r(g^{-\alpha})$, hence 
$$R^pf_*\big(C^{\bullet}_{g^{-\alpha}}(-\lceil G\rceil)\otimes_{\Dmod_Y}\Dmod_{Y\to X}\big)\simeq R^pf_*\big(\Mmod_r(g^{-\alpha})\overset{\derL}{\otimes}_{\Dmod_Y}\Dmod_{Y\to X}\big)\quad\text{and}$$
$$R^pf_*F_k\big(C^{\bullet}_{g^{-\alpha}}(-\lceil G\rceil)\otimes_{\Dmod_Y}\Dmod_{Y\to X}\big)\simeq 
R^pf_*F_k\big(\Mmod_r(g^{-\alpha})\overset{\derL}{\otimes}_{\Dmod_Y}\Dmod_{Y\to X}\big).$$

Finally, by the definition of push-forward for right $\Dmod$-modules we have
$$R^pf_*\big(\Mmod_r(g^{-\alpha})\overset{\derL}{\otimes}_{\Dmod_Y}\Dmod_{Y\to X}\big)\simeq 
H^pf_+\Mmod_r(g^{-\alpha}),$$
and by ($\ref{birational}$) this is $0$ if $p\neq 0$, and is canonically isomorphic to $\Mmod_r (h^{-\alpha})$ if $p=0$. 
The assertions in the proposition follow by combining all these facts.
\end{proof}
 
\begin{remark}[{\bf Local vanishing}]
The statement in Theorem \ref{formula_log_resolution} i) is a generalization of the Local Vanishing theorem 
for multiplier ideals \cite[Theorem 9.4.1]{Lazarsfeld}, in view of the calculation in Proposition \ref{formula_I_0} below.
\end{remark}

As a consequence of the vanishing statements in Theorem~\ref{formula_log_resolution}(i), provided by strictness, we deduce the following local Nakano-type vanishing result, first obtained by Saito \cite[Corollary 3]{Saito-LOG} when $D$ is reduced; cf. Corollary \ref{intro_local_vanishing} in the Introduction and the discussion following it. 

\begin{corollary}\label{cor_vanishing}
Let $D$ be an effective $\QQ$-divisor on the smooth variety $X$ and $f\colon Y\to X$ a log resolution of $(X,D)$ that is an isomorphism over
$X\smallsetminus {\rm Supp}(D)$. If $E= (f^*D)_{{\rm red}}$, then
$$R^q f_*\big(\shO_Y(-\lceil f^*D\rceil)\otimes_{\shO_Y}\Omega_Y^p(\log E)\big)=0\quad\text{for}\quad p+q>n=\dim(X).$$
\end{corollary}

\begin{proof}
We argue by descending induction on $p$, the case $p>n$ being trivial. Suppose now that $p\leq n$ and $q>n-p$.
After possibly replacing $X$ by suitable open subsets, we may assume that $D=\alpha\cdot {\rm div}(h)$. 
We may thus apply Theorem~\ref{formula_log_resolution} to deduce that if
$$C^{\bullet}=F_{-n}\left(C^{\bullet}_{g^{-\alpha}}(-\lceil f^*D\rceil)\otimes_{\Dmod_Y}\Dmod_{Y\to X}\right)[p-n],$$
then
\begin{equation}\label{eq1_cor_vanishing}
R^jf_*C^{\bullet}=0\quad\text{for}\quad j>n-p.
\end{equation}
Note that by definition, we have
$$C^i=\shO_Y(-\lceil f^*D\rceil)\otimes_{\shO_Y}\Omega_Y^{p+i}(\log E)\otimes_{\shO_Y}f^*F_i\Dmod_X\quad\text{for}\quad 0\leq i\leq n-p.$$
Consider the spectral sequence
$$E_1^{i,j}=R^jf_*C^i\Rightarrow R^{i+j}f_*C^{\bullet}.$$
It follows from (\ref{eq1_cor_vanishing}) that $E_{\infty}^{0,q}=0$.
Now by the projection formula we have
\begin{equation}\label{eq2_cor_vanishing}
E_1^{i,j}=R^jf_*\big(\shO_Y(-\lceil f^*D\rceil\otimes_{\shO_Y}\Omega_Y^{p+i}(\log E)\big)\otimes_{\shO_X}F_i\Dmod_X.
\end{equation}
In particular, it follows from the inductive hypothesis that for every $r\geq 1$ we have
$E_1^{r,q-r+1}=0$, hence $E_r^{r,q-r+1}=0$ as well.
On the other hand, we clearly have $E_r^{-r,q+r-1}=0$, since this is a first-quadrant spectral sequence.
We thus conclude that
$$E_r^{0,q}=E_{r+1}^{0,q}\quad\text{for all}\quad r\geq 1,$$
hence $E_1^{0,q}=E_{\infty}^{0,q}=0$. 
Using (\ref{eq2_cor_vanishing}) again, we conclude that
$$R^q f_* \big(\shO_Y(-\lceil f^*D\rceil)\otimes_{\shO_Y}\Omega_Y^p(\log E)\big)=0.$$
\end{proof}

\subsection{The ideal $I_0(D)$ and log canonical pairs}\label{scn:I_0}
We now use Theorem~\ref{formula_log_resolution} in order to relate $I_0(D)$ to multiplier ideals. Recall that for a $\QQ$-divisor $B$, one denotes by $\I (B)$ the associated multiplier ideal; see \cite[Ch.9]{Lazarsfeld} for the definition and basic properties.

\begin{proposition}\label{formula_I_0}
If $f\colon Y\to X$ is a log resolution of $(X,D)$ that is an isomorphism over $X\smallsetminus D$, and  
$E=(f^*D)_{\rm red}$, then
$$
I_0(D) \simeq f_*\shO_Y\big(K_{Y/X}+E-\lceil f^*D\rceil\big)=\I\big((1-\epsilon)D\big)
$$
for $0<\epsilon\ll 1$. 
\end{proposition}

\begin{proof}
The first equality follows from Theorem~\ref{formula_log_resolution}, together with the fact that the term
$F_{-n}C^{\bullet}_{g^{-\alpha}}(-\lceil f^*D\rceil)$ consists of 
$$\omega_Y(E-\lceil f^*D\rceil)$$
placed in degree $0$. 
The second equality then follows from the definition of multiplier ideals and the fact that if $A$ is an effective divisor with support $E$, then
$$E-\lceil A\rceil=-\lfloor (1-\epsilon) A\rfloor\quad\text{for}\quad 0<\epsilon\ll 1.$$
\end{proof}

As in \cite{MP1} in the case of reduced divisors, we obtain therefore that for every $\QQ$-divisor $D$ we have that 
$I_0 (D) = \shO_X$ if and only if the pair $(X, D)$ is log canonical, which leads to the following:

\begin{definition}\label{k-log-can}
The pair $(X,D)$ is \emph{$k$-log canonical} if 
$$I_0 (D) = \cdots = I_k (D) = \shO_X.\footnote{We note that by the results in \cite[\S5]{MP3}, at least in the case of divisors of the form 
$D = \alpha Z$, with $\alpha \in \QQ_{>0}$, this condition is equivalent simply to $I_k (D) = \shO_X$ (cf. Remark~\ref{rmk_inclusion_ideals}).}$$
\end{definition}

Note however that by Remark \ref{old_ideal}, the triviality of any $I_k (D)$ is possible only if $\lceil D \rceil = Z$; in 
general it is more suitable to focus on the triviality of the ideals $I_k' (D)$. We therefore introduce also:

\begin{definition}\label{reduced-k-log-can}
The pair $(X,D)$ is \emph{reduced $k$-log canonical} if 
$$I_0' (D) = \cdots = I_k' (D) = \shO_X,$$
or equivalently 
$$I_0 (D) = \cdots = I_k (D) = \shO_X (Z - \lceil D \rceil).$$
\end{definition}

\begin{example}
Let $Z$ have an ordinary singularity, i.e. an isolated singular point whose projectivized tangent cone is smooth, of multiplicity $m$. If $D = \alpha Z$ with $0 < \alpha \le 1$, then 
$$(X, D) {\rm ~ is~} k {\rm-log ~canonical ~} \iff k \le [\frac{n}{m} - \alpha].$$ 
See Corollary \ref{triviality_ordinary_case} and Remark \ref{ordinary_vfil}.
\end{example}

\section{Local study and global vanishing theorem}

\subsection{Generation level of the Hodge filtration, and examples}\label{scn:level}
As above, we consider a divisor $D=\alpha H$, with $H = {\rm div}(h)$ for some nonzero $h\in\shO_X(X)$ and $\alpha\in\QQ_{>0}$. 
We denote by $Z$ the support of $D$, and $\beta = 1 - \alpha$. 
By construction, the filtration on $\Mmod(h^{\beta})$ is compatible with the order filtration on $\Dmod_X$. This means that for every $k, \ell \geq 0$ we have
\begin{equation}\label{eq1_filtration}
F_\ell \Dmod_X\cdot\big(I_k(D)\otimes \shO_X(kZ + H)h^{\beta}\big)\subseteq I_{k+\ell}(D)\otimes \shO_X ((k+\ell)Z+H)h^{\beta},
\end{equation}
or equivalently for every $k \geq 0$ we have
\begin{equation}\label{eq2_filtration}
F_1 \Dmod_X\cdot\big(I_k(D)\otimes \shO_X(kZ + H)h^{\beta}\big)\subseteq I_{k+1}(D)\otimes \shO_X ((k+1)Z+H)h^{\beta}.
\end{equation}
By working locally, we may assume that we also have an equation $g$ for $Z$.
With this notation, condition (\ref{eq2_filtration}) is equivalent to
the following two conditions: 
\begin{equation}\label{eq0_semicontinuity_question}
g\cdot I_k(D)\subseteq I_{k+1}(D)
\end{equation}
and for every derivation $Q$ of $\shO_X$ and every $w\in I_k(D)$, we have
\begin{equation}\label{eq2_semicontinuity_question}
g\cdot Q(w)-kw\cdot Q(g) - \alpha gw\cdot \frac{Q(h)}{h}\in I_{k+1}(D).
\end{equation}

We now turn to the problem of describing the generation level of the filtration on $\Mmod(h^{\beta})$. Recall that one says that the filtration is generated at level $k$ if
$$F_{\ell}\Dmod_X\cdot F_k\Mmod(h^{\beta})=F_{k+\ell}\Mmod(h^{\beta})\quad\text{for all}\quad \ell\geq 0,$$
or in other words if equality is satisfied in ($\ref{eq1_filtration}$).
This is of course equivalent to having 
$$F_1\Dmod_X\cdot F_p\Mmod(h^{\beta})=F_{p+1}\Mmod(h^{\beta})\quad\text{for all}\quad p\geq k.$$

%Before discussing the level of the filtration on $\Mmod(h^{-\alpha})$, we pose the following:

%\begin{question}\label{question_vanishing}
%With the notation in Proposition~\ref{formula_log_resolution}, is it true that
%$$R^qf_*\big(\shO_Y(-\lceil f^*D\rceil)\otimes_{\shO_Y}\Omega_Y^p(\log E)\big)=0\quad\text{for}\quad p+q>n=\dim(X)?$$
%\end{question}

%\begin{remark}
%Note that the question has a positive answer when $X$ is a surface. Indeed, if $q=0$, then $p>2$, hence $\Omega_Y^p(\log E)=0$. If $q=2$,
%then the vanishing follows from the fact that all fibers of $f$ have dimension $\leq 1$. Finally, if $q=1$ and $p=2$, then 
%$$E-\lceil f^*D\rceil=-\lfloor (1-\epsilon)f^*D\rfloor,$$
%hence 
%$$R^1f_*\big(\omega_Y(E-\lceil f^*D\rceil)\big)=0$$
%by the Local Vanishing theorem. 
%\end{remark}

Suppose now that we are in the setting of Theorem~\ref{formula_log_resolution}. 

\begin{theorem}\label{criterion_level}
The filtration on $\Mmod(h^{\beta})$ is generated at level $k$ if and only if
$$R^qf_*\big(\Omega_Y^{n-q}(\log E)\otimes_{\shO_Y}\shO_Y(-\lceil f^*D\rceil)\big)=0\quad\text{for}\quad q>k.$$
In particular, the filtration is always generated at level $n-1$.
\end{theorem}

\begin{proof}
The proof follows almost verbatim that of \cite[Theorem~17.1]{MP1}. It is more convenient to work equivalently with $\Mmod (h^{-\alpha})$, and 
in fact with the associated right $\Dmod_X$-module $\Mmod_r(h^{-\alpha})$. 
It is enough to show that 
\begin{equation}\label{eq1_criterion_level}
F_{k-n} \Mmod_r(h^{-\alpha})\cdot F_1\Dmod_X=F_{k-n+1}\Mmod_r(h^{-\alpha})
\end{equation}
if and only if 
$$R^{k+1}\big(f_*\Omega_Y^{n-k-1}(\log E)\otimes_{\shO_Y} \shO_Y(-\lceil f^*D\rceil)\big)=0.$$
The inclusion ``$\subseteq$" in (\ref{eq1_criterion_level}) always holds of course by the definition of a filtration,
hence the issue is the reverse inclusion.

With the notation in \S\ref{complex_for_SNC}, for every $p$ let
$$C^{\bullet}_{p}:=F_p\big(C^{\bullet}_{g^{-\alpha}}(-\lceil f^*D\rceil)\otimes_{\Dmod_Y}\Dmod_{Y\to X}\big),$$
where $g=h\circ f$. Consider the 
morphism of complexes
$$\Phi_k\colon C^{\bullet}_{k-n}\otimes_{f^{-1}\shO_X}{f^{-1}F_1\Dmod_X}\longrightarrow C^{\bullet}_{k+1-n}$$
induced by right multiplication, and let $T^{\bullet}={\rm Ker}(\Phi_k)$.
Using Theorem~\ref{formula_log_resolution}, we see that (\ref{eq1_criterion_level}) holds if and only if the morphism
\begin{equation}\label{eq0_generation_filtration}
R^0f_*C^{\bullet}_{k-n}\otimes_{\shO_X}F_1\Dmod_X\longrightarrow R^0f_*C^{\bullet}_{k+1-n}
\end{equation}
induced by $\Phi_k$ is surjective.

For every $m\geq 0$, let $R_m$ be the kernel of the morphism
induced by right multiplication
$$F_m\Dmod_X\otimes_{\shO_X}F_1\Dmod_X\longrightarrow F_{m+1}\Dmod_X.$$
Note that this is a surjective morphism of locally free $\shO_X$-modules, hence $R_m$ is a locally free $\shO_X$-module
and for every $p$ we have
$$T^p=\shO_Y(-\lceil f^*D\rceil)\otimes_{\shO_Y}\Omega_Y^{n+p}(\log E)\otimes_{f^{-1}\shO_X}f^{-1}R_{k+p}.$$

Consider the first-quadrant hypercohomology spectral sequence
$$E_1^{p,q}=R^qf_*T^{p-n}\implies R^{p+q-n}f_*T^{\bullet}.$$
The projection formula gives
$$R^qf_*T^{p-n}\simeq R^qf_*\big(\shO_Y(-\lceil f^*D\rceil)\otimes_{\shO_Y} \Omega_Y^p(\log E)\big)\otimes_{\shO_X}R_{k+p-n},$$
and this vanishes for $p+q>n$ by Corollary~\ref{cor_vanishing}. We thus deduce from the spectral sequence
that $R^jf_*T^{\bullet}=0$ for all $j>0$.

We first consider the case when $k\geq n$ and show that (\ref{eq1_criterion_level}) always holds.
Indeed, in this case $\Phi_k$ is surjective.
It follows from the projection formula and the long exact sequence in cohomology that we have an exact sequence
$$R^0f_*C^{\bullet}_{k-n}\otimes_{\shO_X}F_1\Dmod_X\to R^0f_*C^{\bullet}_{k+1-n}\to R^1f_*T^{\bullet}.$$
We have seen that $R^1f_*T^{\bullet}=0$, hence  the morphism in (\ref{eq0_generation_filtration}) is surjective.

Suppose now that $0\leq k<n$. Let $B^{\bullet}\hookrightarrow C^{\bullet}_{k+1-n}$ be the subcomplex given by
$B^p=C_{k+1-n}^p$ for all $p\neq -k-1$ and $B^{-k-1}=0$. Note that we have a short exact sequence of complexes
\begin{equation}\label{eq2_criterion_level}
0\longrightarrow B^{\bullet}\longrightarrow C^{\bullet}_{k+1-n}\longrightarrow C_{k+1-n}^{-k-1}[k+1]\longrightarrow 0.
\end{equation}
It is clear that $\Phi_k$ factors as
$$C^{\bullet}_{k-n}\otimes_{f^{-1}\shO_X}{f^{-1}F_1\Dmod_X}\overset{\Phi'_k}\longrightarrow B^{\bullet}\hookrightarrow C^{\bullet}_{k+1-n}.$$
Moreover, $\Phi'_k$ is surjective and ${\rm Ker}(\Phi'_k)=T^{\bullet}$. As before, since $R^1f_*T^{\bullet}=0$, we conclude that morphism induced
by $\Phi'_k$:
$$R^0f_*C^{\bullet}_{k-n}\otimes_{\shO_X}F_1\Dmod_X\to R^0f_*B^{\bullet}$$ is surjective. This implies that (\ref{eq0_generation_filtration}) is surjective
if and only if the morphism 
\begin{equation}\label{eq3_generation_filtration}
R^0f_*B^{\bullet}\to R^0f_*C^{\bullet}_{k+1-n}
\end{equation}
is surjective. The exact sequence (\ref{eq2_criterion_level}) induces an exact sequence
$$R^0f_*B^{\bullet}\to R^0f_*C^{\bullet}_{k+1-n}\to R^{k+1}f_*C_{k+1-n}^{-k-1}\to R^1f_*B^{\bullet}.$$
We have seen that $R^2f_*T^{\bullet}=0$, and we also have 
$$R^1f_*\big(C^{\bullet}_{k-n}\otimes_{f^{-1}\shO_X}{f^{-1}F_1\Dmod_X}\big)=0.$$
This follows as above, using the projection formula, the hypercohomology spectral sequence, and Corollary~\ref{cor_vanishing}.
We deduce from the long exact sequence associated to 
$$0\longrightarrow T^{\bullet}\longrightarrow C^{\bullet}_{k-n}\otimes_{f^{-1}\shO_X}{f^{-1}F_1\Dmod_X}\longrightarrow B^{\bullet}\longrightarrow 0$$
that $R^1f_*B^{\bullet}=0$. Putting all of this together, we conclude that (\ref{eq0_generation_filtration})
is surjective if and only if $R^{k+1}f_*C_{k+1-n}^{-k-1}=0$.
Since  by definition we have
$$R^{k+1}f_*C_{k+1-n}^{-k-1}=R^{k+1}f_*\big(\shO_Y(-\lceil f^*D\rceil)\otimes_{\shO_Y}\Omega_Y^{n-k-1}(\log E)\big),$$
this completes the proof of the first assertion in the proposition. The second assertion follows from the first, since all fibers of $f$ have dimension $<n$.
\end{proof}

\begin{example}[{\bf Nodal curves}]\label{example_node}
If $X$ is a smooth surface and $Z$ is a reduced curve on $X$, defined by $h\in\shO(X)$, such that $Z$ has a node at $x\in X$ and no other singularities, then
the filtration on $\Mmod(h^{\beta})$ is generated at level $0$. Indeed, let $f\colon Y\to X$ be the blow-up
of $X$ at $x$, with exceptional divisor $F$. This is a log resolution of $(X,Z)$, hence our assertion follows if we show that
\begin{equation}\label{eq_example_node}
R^1f_*\big(\Omega_Y(\log E)\otimes_{\shO_Y}\shO_Y(-\lceil \alpha f^*Z\rceil)\big)=0,
\end{equation}
where $E=\widetilde{Z}+F$. Note that $f^*Z=\widetilde{Z}+2F$ and
we may assume that $0<\alpha\leq 1$. If $\frac{1}{2}<\alpha\leq 1$, then $\lceil \alpha f^*Z\rceil =f^*Z$ and (\ref{eq_example_node}) 
follows from \cite[Theorem~B]{MP1} using the projection formula. On the other hand, if $0<\alpha\leq \frac{1}{2}$, then 
$\lceil \alpha f^*Z\rceil=E$ and the vanishing follows from the fact that the pair $(X,Z)$ is log canonical, using \cite[Theorem~14.1]{GKKP} (though, in this case, one could also check this directly).

Once we know that the filtration on $\Mmod(h^{\beta})$ is generated at level $0$, it is straightforward to check that 
$$I_k(\alpha Z)={\mathfrak m}_x^k, \,\,\,\,\,{\rm for~all}\,\,\,\,\,\,0<\alpha\leq 1 {\rm ~and~} k \ge 0,$$
where ${\mathfrak m}_x$ is the ideal defining $x$ in $X$.
\end{example}

Unlike in the case when $D$ is a reduced integral divisor, when the filtration $F_{\bullet} \shO_X(*D)$ is generated at level $n-2$ by \cite[Theorem~B]{MP1}, 
in general it is not possible to improve the bound given by Proposition~\ref{criterion_level}.

\begin{example}[{\bf Optimal generation level}]
It can happen that on a surface $X$ the filtration on $\Mmod(h^{\beta })$ is not generated at level $0$.
 Suppose, for example, that $X=\AAA^2$ and $Z=L_1+L_2+L_3$, where $L_1$, $L_2$, and $L_3$ are 
3 lines passing through the origin. 
 If $f\colon Y\to X$ is the blow-up of the origin
 and $E=(f^*Z)_{\rm red}$, then we write $E=F+G_1+G_2+G_3$, where $F$ is the exceptional divisor and the $G_i$ are the strict transforms of the $L_i$.
 Let $D=\alpha Z$ with $0<\alpha\ll 1$, so that $\lceil f^*D\rceil=E$.
 If 
 $$H^1\big(Y,\Omega_Y(\log E)\otimes_{\shO_Y}\shO_Y(-\lceil f^*D\rceil)\big)=H^1\big(Y,\Omega_Y(\log E)\otimes_{\shO_Y}\shO_Y(-E)\big)$$
 were zero, then it would follow from the standard exact sequence
 $$0\to\Omega_Y(\log E)\otimes_{\shO_Y}\shO_Y(-E)\to\Omega_Y\to\Omega_F\oplus\bigoplus_{i=1}^3\Omega_{G_i}\to 0$$
 that the map 
 $$H^0(Y,\Omega_Y)\to H^0(F,\Omega_F)\oplus\bigoplus_{i=1}^3H^0(G_i,\Omega_{G_i})$$
 is surjective. In particular, we would deduce that the map
 $$H^0(X,\Omega_X)\to\bigoplus_{i=1}^3H^0(L_i,\Omega_{L_i})$$
 is surjective. It is an easy exercise to see that this is not the case. Note that the non-vanishing of $H^1\big(Y,\Omega_Y(\log E)\otimes_{\shO_Y}\shO_Y(-E)\big)$ is not inconsistent with the Steenbrink-type vanishing in \cite[Theorem~14.1]{GKKP}, since the pair $(X, Z)$ is not log-canonical.
 \end{example}

\begin{example}[{\bf Quasi-homogeneous isolated singularities}]\label{quasihomogeneous}
For the class of quasi-homogeneous isolated singularities (such as those in the examples above), 
the generation level for the filtration on $\Mmod(h^{\beta})$ can be detected by the Bernstein-Sato polynomial. Before formulating this more precisely, we recall
some definitions. Suppose that $Z$ is a hypersurface in $X$ defined by $h\in\shO_X(X)$. The \emph{Bernstein-Sato polynomial}  of $Z$ is 
the non-zero monic polynomial $b_{h}\in\CC[s]$ of smallest
degree such that we locally
have a relation of the form
$$b_{h}(s)h^s=P(s)\bullet h^{s+1}$$
for some nonzero $P\in\Dmod_X[s]$. If $Z$ is non-empty, it is known that $(s+1)$ divides $b_{h}$; moreover, all the roots of $b_{h}$ are negative rational numbers. In this case, one defines $\widetilde{\alpha}_{h}=-\lambda$, where $\lambda$ is the largest root of the \emph{reduced Bernstein-Sato polynomial} $\widetilde{b}_{h}=b_{h}(s)/(s+1)$. Note that $\widetilde{b}_{h}$ has degree $0$ if and only if $Z$ is smooth, and in this case one makes the convention that $\widetilde{\alpha}_{h}=\infty$.

The statement is that if $Z = {\rm div}(h)$ is reduced and has a unique singular point at $x$, which is a quasi-homogeneous singularity, 
and $D = \alpha Z$, then the generation level $k_0$ of the filtration on $\Mmod(h^{\beta})$ (i.e. the smallest $k$ such that the filtration is generated at level $k$)  is  
$$k_0 =  \lfloor n-\widetilde{\alpha}_{h}-\alpha\rfloor.$$ 
This was proved by Saito \cite[Theorem~0.7]{Saito-HF} when $D$ is reduced, i.e. for  $\alpha=1$, and was extended to the general case by Zhang \cite{Zhang}.\footnote{Moreover, based on calculations of Saito,  
Zhang shows in \emph{loc. cit.} that all Hodge ideals of $\QQ$-divisors associated to such singularities can be computed explicitly.}

Note that for such singularities there is an explicit formula for $\widetilde{\alpha}_{h}$; see e.g. \cite[\S4.1]{Saito-HF}. Just as an illustration, for $h = xy (x+y)$, which describes the previous example, we have $\widetilde{\alpha}_{h} = 2/3$, and so for 
$\alpha$ small (more precisely $0 < \alpha \le 1/3$)  we recover the fact that the generation level is equal to $1$.
\end{example}

\begin{example}[{\bf Incomensurability of higher Hodge ideals}]\label{parameter_dependence}
Suppose that $X$ is a smooth surface and $Z=\sum_{i=1}^rD_i$ is a reduced effective divisor on $X$.
Let $f\colon Y\to X$ be a log resolution of $(X,Z)$ that is an isomorphism over $X\smallsetminus Z$, and put $E=(f^*Z)_{\rm red}$.
Let  $D=\sum_{i=1}^r(1-a_i)D_i$ be a divisor with $0\leq a_i\ll 1$ for all $i$, so that $\lceil f^*D\rceil=f^*Z$. 
In this case we have
$$R^1f_*\big(\shO_Y(-\lceil f^*D\rceil)\otimes \Omega_Y^1(\log E)\big)=0$$
by the projection formula and \cite[Theorem~B]{MP1}, and so the filtration is generated at level $0$.
It follows from the discussion at the beginning of the section (see ($\ref{eq0_semicontinuity_question}$) and 
($\ref{eq2_semicontinuity_question}$))
that if $g$ is a local equation of $Z$,  and $D=\alpha Z$, with $\alpha\leq 1$ and close to $1$, 
then
$I_{k+1}(D)$ is generated by $g\cdot I_k(D)$ and
$$\left\{h\cdot Q(w)-(\alpha+k)w\cdot Q(h)\mid w\in I_k(D),Q\in {\rm Der}_{\CC}(\shO_X)\right\}.$$

For example, if $X=\CC^2$ and $Z$ is the cusp defined by $x^2+y^3$, then for $D = \alpha Z$ with $\alpha\leq 1$ and close to $1$ we have
$$I_0(D)=(x,y), \,\,I_1(D)=(x^2,xy,y^3),\,\,\text{and}$$
$$I_2(D)=(x^3,x^2y^2,xy^3, y^4-(2\alpha+1)x^2y).$$

Note in particular that if $D_1 = \alpha_1 Z$ and $D_2 = \alpha_2 Z$, with $\alpha_1 < \alpha_2$ both close to $1$, then there is no 
inclusion between the  ideals $I_2 (D_1)$ and $I_2 (D_2)$. This is in contrast with the picture for multiplier ideals, where for any $\QQ$-divisors
$D_1 \le D_2$ one has $I_0 (D_2) \subseteq I_0 (D_1)$; see \cite[Proposition~9.2.32(i)]{Lazarsfeld}. It is not 
hard to check however that 
$$I_2 (D_1) = I_2 (D_2) \,\,\,\,\,{\rm mod} \,\,\,\,x^2 + y^3,$$
and that this is part of a general phenomenon where the picture is well behaved after modding out by a defining equation
for the hypersurface; this follows from the connection with the $V$-filtration, see \cite[Corollary B]{MP3}. 
\end{example}

\begin{remark}\label{rem_multiplicity}
If the filtration is generated at level $k$, then $I_{k+1} (D)$ is generated by the terms appearing on the left hand side of 
conditions (\ref{eq0_semicontinuity_question}) and (\ref{eq2_semicontinuity_question}).  A simple calculation shows then that in this case, 
for every $j\geq 1$ and every $x\in X$, we have
$${\rm mult}_xI_{k+j}(D)\geq {\rm mult}_xI_{k+j-1}(D)+{\rm mult}_x Z-1.$$
In particular, we have
$${\rm mult}_xI_{k+j}(D)\geq {\rm mult}_xI_k(D)+j\cdot ({\rm mult}_x Z -1).$$
Since the filtration is always generated at level $n-1$ by Proposition~\ref{criterion_level}, we obtain the following consequence.
\end{remark}

\begin{corollary}\label{cor_singular_nontrivial}
If $D$ is an effective $\QQ$-divisor on the smooth variety $X$, with support $Z$, and if $Z$ is singular at some $x\in X$,
then $I_j(D)_x\neq\shO_{X,x}$ for all $j\geq n$. In fact, if $m={\rm mult}_x Z$, then
$${\rm mult}_xI_j(D)\geq (j-n+1)(m-1)\quad\text{for all}\quad j\geq n.$$ 
\end{corollary}

\subsection{Non-triviality criteria}\label{nontriviality}
The following is the analogue of \cite[Theorem~18.1]{MP1} in the setting of $\QQ$-divisors. Let $D$ be an effective 
$\QQ$-divisor on the smooth variety $X$, with $Z={\rm Supp}(D)$, 
and let $\varphi\colon X_1\to X$ be  a projective morphism with $X_1$ smooth, such that  $\varphi$ is an isomorphism 
over $X\smallsetminus Z$.  We denote 
$$Z_1=(\varphi^*Z)_{\rm red} \,\,\,\,\,\,{\rm and} \,\,\,\,\,\, T_{X_1/X}={\rm Coker}(T_{X_1}\to \varphi^*T_X).$$ 

\begin{theorem}\label{smaller_ideal}
With the above notation, the following hold:
\begin{enumerate}
\item[i)] We have an inclusion
$$\varphi_*\big(I_k(\varphi^*D)\otimes_{\shO_{X_1}}\shO_{X_1}(K_{X_1/X}+ k(Z_1-\varphi^*Z))\big)\subseteq I_k(D).$$
\item[ii)] If $J$ is a coherent ideal on $X$ such that $J\cdot T_{X_1 /X}=0$, then 
$$J^k\cdot I_k(D)\subseteq \varphi_*\big(I_k(\varphi^*D)\otimes_{\shO_{X_1}}\shO_{X_1}(K_{X_1/X}+ k (Z_1-\varphi^*Z))\big).$$
\end{enumerate}
\end{theorem}

\begin{proof}
We may assume that $D=\alpha\cdot {\rm div}(h)$, for some $\alpha\in\QQ_{>0}$ and some nonzero $h\in\shO_X(X)$. 
Let $\psi\colon Y\to X_1$ be a log resolution of $(X_1,\varphi^*D)$ that is an isomorphism over $X_1\smallsetminus\varphi^{-1}(Z)$.
We put 
$$f=\varphi\circ\psi \,\,\,\,\,\,{\rm and} \,\,\,\,\,\, E=(f^*Z)_{\rm red}.$$
With the notation in 
\S\ref{complex_for_SNC}, consider the filtered complex $C^{\bullet}=C^{\bullet}_{g^{-\alpha}}(-\lceil f^*D\rceil)$, where 
$g=h\circ f$. We have an inclusion of complexes
$$A^{\bullet}=C^{\bullet}\otimes_{\Dmod_Y}\Dmod_{Y\to X_1}\hookrightarrow B^{\bullet}=C^{\bullet}\otimes_{\Dmod_Y}\Dmod_{Y\to X}.$$
Note that this is an injection due to the fact that $\shO_Y(-\lceil f^*D\rceil)$ and $\Omega_Y^q(\log E)$ are locally free sheaves of
$\shO_Y$-modules, while all the maps $F_p\Dmod_{Y\to X_1}\to F_p\Dmod_{Y\to X}$ are generically injective morphisms of locally 
free $\shO_Y$-modules. Consider, for any integer $k$, the short exact sequence of complexes
$$0\longrightarrow F_{k-n} A^{\bullet}\longrightarrow F_{k-n} B^{\bullet} \longrightarrow M_{\bullet}\longrightarrow 0.$$
Applying $\derR f_*$ and taking the corresponding long exact sequence, we obtain a short exact sequence
$$R^0f_*F_{k-n} A^{\bullet}\overset{\iota}\longrightarrow R^0f_*F_{k-n} B^{\bullet}\longrightarrow R^0f_*M^{\bullet}.$$
If $\beta = 1 - \alpha$, it follows from Theorem~\ref{formula_log_resolution} that
$$\derR f_*F_{k-n} B^{\bullet}=R^0f_*F_{k-n} B^{\bullet}\simeq h^{\beta}\shO_X\big(K_X+ kZ + H\big) \otimes_{\shO_X}
I_k(D)$$
and 
$$\derR g_*F_{k-n} A^{\bullet}=R^0g_*F_{k-n} A^{\bullet}\simeq h^{\beta}\shO_{X_1}\big(K_{X_1}+k Z_1 + \varphi^* H\big)\otimes_{\shO_{X_1}} I_k(\varphi^*D).$$
Therefore, after tensoring by $\shO_X (-H)$, the map $\iota$ induces a map
\begin{equation}\label{eq1_smqller_ideal}
\varphi_*\big(I_k(\varphi^*D)\otimes\shO_{X_1}(K_{X_1}+k Z_1)\big)\rightarrow
 I_k(D)\otimes_{\shO_X}\shO_X\big(K_X+k Z\big).
 \end{equation}
Finally, the map $\iota$ is compatible with restriction to open subsets of $X$. By restricting to an open subset $X_0$ in the complement of $Z$,
such that $f$ is an isomorphism over $X_0$, we see that the map in (\ref{eq1_smqller_ideal}) is the identity on $\omega_{X_0}$. We thus deduce the assertion in $i)$ by
tensoring (\ref{eq1_smqller_ideal}) with $\shO_X\big(-K_X-k Z\big)$.
Furthermore, we see that the assertion in $ii)$ follows if we show that $J^k\cdot R^0f_*M^{\bullet}=0$. Since
$$M^p=\shO_Y(-\lceil f^*D\rceil)\otimes_{\shO_Y}\Omega_Y^{n+p}(\log E)\otimes_{\shO_Y}\psi^*(\varphi^*F_{k+p}\Dmod_{X}/F_{k+p}\Dmod_{X_1}),$$
it is enough to show that under our assumption we have
$$\varphi^*F_j\Dmod_X\cdot J^j\subseteq F_j\Dmod_{X_1}\quad\text{for all}\quad j\geq 0.$$
This is proved in \cite[Lemma~18.6]{MP1}.
\end{proof}

We first use Theorem \ref{smaller_ideal} in order to give a triviality criterion for Hodge ideals in terms of invariants 
of a fixed resolution of singularities. We use this in turn in order to bound the largest root of the reduced Bernstein-Sato polynomial (i.e. $\widetilde{\alpha}_{h}$ defined in Example \ref{quasihomogeneous}) in terms of such invariants, in 
\cite[Corollary D]{MP3}.

\begin{proposition}\label{triviality_criterion}
Let $Z$ be  a  reduced divisor on the smooth variety $X$, and let $D = \alpha Z$, with $\alpha \in \QQ_{>0}$.
Let $f \colon Y \to X$ be a log resolution of $(X, Z)$ that is an isomorphism over $X\smallsetminus Z$
and such that the strict transform $\widetilde{Z}$ of $Z$ is smooth. We define
integers $a_i$ and $b_i$ by the expressions
$$f^*Z = \widetilde{Z} + \sum_{i=1}^m a_iF_i\quad\text{and}\quad
K_{Y/X} = \sum_{i=1}^m b_iF_i,$$
where $F_1,\ldots,F_m$ are the prime exceptional divisors. If 
\begin{equation}\label{eq_triviality_criterion}
\frac{b_i + 1}{a_i} \ge k + \alpha\quad\text{for}\quad 1 \le i \le m,
\end{equation}
then $I_k (D) = \shO_X \big((1 - \lceil \alpha\rceil) Z\big)$. In particular, if $0 < \alpha \le 1$, then $I_k (D) = \shO_X$.
\end{proposition}
\begin{proof}
If $D'=\alpha'Z$, where $\alpha'=\alpha+1-\lceil\alpha\rceil$, then it follows from 
Lemma \ref{periodicity} that $I_k(D)=I_k(D')\otimes\shO_X\big((1 - \lceil \alpha\rceil) Z\big)$. 
Since the inequalities (\ref{triviality_criterion}) clearly also hold if we replace $\alpha$ by $\alpha'$,
it follows that it is enough to treat the case 
$0 < \alpha \le 1$.

First, note that since $f^*D$ has simple normal crossings, by Proposition \ref{Hodge_ideals_SNC} we have 
$$I_k (f^* D) = I_k (E) \otimes \shO_Y \big( \sum_{i=1}^m (1 - \lceil \alpha a_i \rceil) F_i\big),$$
where $E = (f^* Z)_{{\rm red}}=\widetilde{Z}+\sum_{i=1}^mF_i$. We apply Theorem \ref{smaller_ideal} i) to obtain the inclusion
\begin{equation}\label{eqn:inclusion}
f_* \big(I_k(E) \otimes \shO_Y (F)  \big) \hookrightarrow I_k(D),
\end{equation}
where
$$F : = \sum_{i=1}^m \big(b_i + k + 1 - k a_i - \lceil \alpha a_i \rceil\big) F_i.$$
On the other hand, since $E = \tilde Z + \sum_{i=1}^m F_i$ has simple normal crossings and $\widetilde{Z}$ is smooth,  
it follows from the description of Hodge ideals of simple normal crossing divisors in \cite[Proposition~8.2]{MP1} that we have 
$$\shO_Y (-k\cdot \sum_{i=1}^m F_i) \subseteq I_k (E).$$
Note that the inequalities in (\ref{eq_triviality_criterion}) imply
 $b_i + 1 \ge ka_i + \lceil \alpha a_i \rceil$ for all $i$, hence the divisor $F-k\cdot\sum_{i=1}^mF_i$ is effective
We thus deduce using 
(\ref{eqn:inclusion}) that we have
$$\shO_X = f_* \shO_Y \hookrightarrow I_k (D).$$
\end{proof}

\begin{remark}
More generally, suppose that we write $Z = \sum_{j=1}^r Z_j$, and consider an effective $\QQ$-divisor 
$D = \sum_{j=1}^r \alpha_j Z_j$ supported on $Z$. For simplicity, let us assume that 
$0<\alpha_j\leq 1$ for all $j$. If $f$ is a log resolution as in Proposition \ref{triviality_criterion}, and we write
$$f^* Z_j = \widetilde{Z_j} + \sum_{i=1}^m a_i^j F_i$$ 
for all $j$ (so that $a_i=\sum_{j=1}^ra_i^j$), then the same proof gives $I_k(D)=\shO_X$ if 
$$b_i + 1 \ge ka_i + \sum_{i=1}^m \alpha_j a_i^j \,\,\,\,\,\,\text{for all}~i.$$
\end{remark}

\medskip

We now turn our attention to non-triviality criteria for the Hodge ideals $I_k (D)$ in terms of the multiplicity of 
$D$, and of its support $Z$, along a given subvariety.

\begin{corollary}\label{inclusion_power}
Let $D$ be an effective $\QQ$-divisor on the smooth variety $X$, and let $Z$ be the support of $D$. 
If $W$ is an irreducible closed subset of $X$ of codimension $r$ such that ${\rm mult}_WZ=a$ and 
${\rm mult}_WD=b$, and if $q$ is a non-negative integer such that
$$b+ka>q+r+2k-1,$$
then $I_k(D)\subseteq I_W^{(q)}$, the $q$-th symbolic power of $I_W$.
In particular, if 
$${\rm mult}_WD>\frac{q+r+2k-1}{k+1},$$
then $I_k(D)\subseteq I_W^{(q)}$. 
\end{corollary}

\begin{proof}
After possibly restricting to a suitable open subset of $X$ meeting $W$, we may assume that $W$ is smooth. The first assertion in the corollary
follows by applying Theorem \ref{smaller_ideal}(ii) to the blow-up $\varphi\colon X_1\to X$ along $W$. Note that we may take $J=I_W$ by \cite[Example~18.7]{MP1}, while $I_k (\varphi^*D) \subseteq \shO_{X_1} (Z_1 - \lceil \varphi^*D \rceil)$. The last assertion follows thanks to the fact that by assumption we have $a\geq b$. 
\end{proof}

\begin{remark}
An interesting consequence of the above corollary is that if $Z$ is a reduced divisor on the smooth,
$n$-dimensional variety $X$, $k$ is a positive integer, and
$x\in X$ is a point such that 
$${\rm mult}_xZ\geq 2+\frac{n}{k},$$
then $I_k(D)$ is non-trivial at $x$ for every effective $\QQ$-divisor $D$ with support $Z$ (no matter how small the coefficients).
\end{remark}

\begin{example}[{\bf Ordinary singularities, I}]\label{eg_ordinary1}
Let $X$ be a smooth variety of dimension $n$, and $Z$ a reduced divisor with an ordinary singularity 
at $x\in X$ (recall that this means that the projectivized tangent cone of $Z$ at $x$ is smooth), for instance a 
cone over a smooth hypersurface. 
If $D=\alpha Z$, with $\alpha$ a rational number satisfying $0<\alpha\leq 1$, then
$${\rm mult}_xZ\leq\frac{n}{k+\alpha} \implies I_k(D)_x=\shO_{X,x}.$$
Note that the converse of this statement will be proved in Corollary \ref{triviality_ordinary_case}  below.

Indeed, the assumption implies that after possibly replacing
$X$ by an open neighborhood of $x$, the blow-up
$f\colon Y\to X$ of $X$ at $x$ gives a log resolution of $(X,Z)$. Let $E=F+\widetilde{Z}$, where $F$ is the exceptional divisor of $f$ and $\widetilde{Z}$
is the strict transform of $Z$.
If $m={\rm mult}_x Z$, then we deduce from Theorem~\ref{smaller_ideal}
that
$$f_*\big(I_k(\varphi^*D)\otimes_{\shO_Y}\shO_Y(K_{Y/X}+k(E - f^*Z))\big)$$
$$=f_*\big(I_k(\varphi^*D)\otimes_{\shO_Y}\shO_Y((n-1+k-km)F)\big)\subseteq I_k(D).$$
Now since $\varphi^*D$ is supported on the simple normal crossings divisor $E$, by Proposition \ref{Hodge_ideals_SNC} we have 
$$I_k (\varphi^*D) = I_k (E) \otimes_{\shO_Y} \shO_Y ((1- \lceil \alpha m\rceil) F),$$
where we use the fact that $\lceil \alpha \rceil = 1$.
Moreover, by \cite[Proposition~8.2]{MP1} we have 
$$I_k(E)=\big(\shO_Y(-\widetilde{Z})+\shO_Y(-F)\big)^k\supseteq \shO_Y(-kF).$$
Now by assumption
$$n-km-\lceil \alpha m\rceil\geq 0,$$
hence we deduce $I_k(D)=\shO_X$. 
\end{example}

\begin{example}[{\bf Ordinary singularities, II}]\label{eg_ordinary2}
With considerable extra work, one can say more in the ordinary case. We keep the notation of the 
previous example, and assume that $x$ is a singular point of $Z$, hence $m\geq 2$. 
If $k$ is a positive integer such that
$$(k-1)m +\lceil \alpha m\rceil<n\quad\text{and}\quad k\leq n-2,$$
then we have
$$I_k(D)={\mathfrak m}_x^{km +\lceil \alpha m \rceil-n}$$
in a neighborhood of $x$, where ${\mathfrak m}_x$ is the ideal defining $x$
(with the convention that ${\mathfrak m}_x^j=\shO_X$ if $j\leq 0$).
The argument is similar to that in \cite[Proposition~20.7]{MP1}, so we omit it.
\end{example}

%One can use the examples above to deduce the following analogue of \cite[Theorem~D]{MP1}, as in \cite[Example~20.10]{MP1}. The result 
%also follows however from Corollary 0.33 in the other note, as in this case $\tilde{\alpha}_f = \frac{n}{m}$.

In what follows we make use of some general properties of Hodge ideals that will be proved in
Ch.\ref{section_restriction}, namely the Restriction and Semicontinuity Theorems.

\begin{corollary}\label{triviality_ordinary_case}
If $X$ is a smooth $n$-dimensional variety, $Z$ is a reduced divisor with an ordinary singularity of multiplicity $m\geq 2$ at $x\in X$, and 
$D=\alpha Z$ with $0<\alpha\leq 1$, then 
$$I_k (D)_x = \shO_{X,x} \iff m \le \frac{n}{k + \alpha}.$$
\end{corollary}

\begin{proof}
The ``if" part follows directly from Example~\ref{eg_ordinary1}. For the converse, we need to show that if
${\mathfrak m}_x$ is the ideal defining $x$ and $m > \frac{n}{k + \alpha}$, then $I_k(D)\subseteq {\mathfrak m}_x$.
We may assume that $Z$ is defined in $X$ by $h\in\shO_X(X)$.
Let $r\geq 0$ be such that $n+r=mk+\lceil m\alpha\rceil-1$ and consider the divisor $Z'$ in $X\times {\mathbf C}^r$
defined by $h+y_1^m+\cdots+y_r^m$, where $y_1,\ldots,y_r$ are the coordinates on ${\mathbf C}^r$. 
It is easy to check that $Z'$ is reduced and has an ordinary singularity at $(x,0)$.
By the Restriction Theorem
(see Theorem~\ref{restriction} and Remark \ref{restriction_higher_codimension}  below), we have $I_k(\alpha Z)\subseteq I_k(\alpha Z')\cdot\shO_X$, where we consider $X$  embedded in
$X\times {\mathbf C}^r$ as $X\times\{0\}$. After replacing $X$ and $Z$ by $X'$ and $Z'$, we may thus assume that
$n=mk+\lceil m\alpha\rceil-1$. If $k\leq n-2$, then we may apply Example~\ref{eg_ordinary2} to conclude that $I_k(D)\subseteq {\mathfrak m}_x$.
Otherwise we have 
$$k\geq n-1=mk+\lceil m\alpha\rceil-2,$$
which easily implies $m=2$, $k=1$, and $\alpha\leq\frac{1}{2}$, hence $n=2$. Since $Z$ has an ordinary singularity at $x$, it follows that it must be a node, and in this case we have $I_1(\alpha Z)={\mathfrak m}_x$ by  Example~\ref{example_node}.
\end{proof}

\begin{remark}\label{ordinary_vfil}
One can give an alternative argument, arguing as follows. 
Suppose that $Z$ is a reduced divisor in $X$, defined by $h\in\shO_X(X)$.
It is shown in \cite[Corollary~C]{MP3} that for $0<\alpha\leq 1$, we have
$I_k(\alpha Z)=\shO_X$ if and only if $k\leq\widetilde{\alpha}_h-\alpha$.
If $Z$ has an ordinary singularity at $x\in X$, of multiplicity $m\geq 2$, then after replacing $X$ by a suitable
neighborhood of $x$, we have $\widetilde{\alpha}_h=\frac{n}{m}$
(see \cite[\S2.5]{Saito-MLCT}), and we recover the assertion in Corollary~\ref{triviality_ordinary_case}.
\end{remark}

\begin{question}\label{better_bound}
Is it true that if $X$ is a smooth $n$-dimensional variety, $Z$ is a reduced divisor on $X$,
$D$ is an effective $\QQ$-divisor with support $Z$,
and for a point $x\in Z_{\rm sing}$ we have
$$k\cdot {\rm mult}_x Z + {\rm mult}_x D >n,$$
then $I_k(D)\subseteq {\mathfrak m}_x$? 
\end{question}

This would be a natural improvement of Corollary \ref{inclusion_power}, 
and it does hold when $D$ is reduced by \cite[Corollary 21.3]{MP1}. 
We may of course assume that $\lceil D\rceil =Z$, since otherwise the inclusion is trivial (see Remark~\ref{old_ideal}).
At the moment we have:

\begin{corollary}\label{good_bound}
Question \ref{better_bound} has a positive answer if $D$ is of the form $D=\alpha Z$. 
\end{corollary}
\begin{proof}
We may assume that $\alpha\leq 1$ and, arguing as in
the proof of \cite[Theorem~E]{MP1}, we construct a reduced divisor $F$ on $X\times U$, for a smooth variety $U$, such that for  $t\in U$ general the divisor $F_t=F\vert_{X\times\{t\}}$ is reduced, with an ordinary singularity at $(x,t)$ of multiplicity $m={\rm mult}_xZ$,
 and for some $t_0\in U$, the isomorphism
$X\simeq X\times\{t_0\}$ maps $D$ to $F_{t_0}$. In this case Corollary~\ref{triviality_ordinary_case} implies that $I_k(F_t)$ vanishes at $(x,t)$
for $t\in U$ general, and the Semicontinuity Theorem (see Theorem~\ref{semicontinuity} below) implies that $I_k(F_{t_0})$ vanishes at $(x,t_0)$.
\end{proof}

This allows us in particular to provide an analogue of \cite[Theorem~A]{MP1}:

\begin{corollary}\label{smoothness_equivalence}
If $D$ is of the form $D = \alpha Z$, then 
$$Z {\rm ~is ~smooth ~} \iff I_k (D) = \shO_X (Z - \lceil D \rceil )\,\,\,\,\,\,{\rm for ~all~}k.$$
\end{corollary}
\begin{proof}
It suffices to assume $0 < \alpha \le 1$, in which case the condition becomes $I_k (D) = \shO_X$ for all $k$.
By Corollary \ref{good_bound} however, if ${\rm mult}_x Z \ge 2$, then $I_k (D) \subseteq \frak{m}_x$ for all
$k > \frac{n}{2} - \alpha$.
\end{proof}

\subsection{Vanishing theorem}\label{vanishing}
As usual, we consider an effective ${\mathbf Q}$-divisor $D$ with support $Z$, on the smooth variety $X$.
In this section we assume that $X$ is projective, and prove a vanishing theorem for Hodge ideals, extending 
\cite[Theorem F]{MP1} as well as Nadel Vanishing for $\QQ$-divisors.

We start by choosing a positive integer $\ell$ such that
$\ell D$ is an integral divisor, and further assume that 
there exists a line bundle $M$ on $X$ such that 
\begin{equation}\label{root}
M^{\otimes \ell} \simeq \shO_X (\ell D),
\end{equation}
so that the setting of \S\ref{global_setting} applies. We note that 
this can always be achieved after passing to a finite flat cover of $X$.

\begin{theorem}\label{vanishing_Hodge_ideals}
Let $X$ be a smooth projective variety of dimension $n$ and $D$ an effective $\QQ$-divisor on $X$ such that 
$(\ref{root})$ is satisfied. Let $L$ be a line bundle on  $X$ such that $L+Z - D$ is ample. For some $k \ge 0$, assume that the pair $(X,D)$ 
is reduced $(k-1)$-log-canonical, i.e. $I_0 (D) = \cdots = I_{k-1} (D) = \shO_X (Z - \lceil D \rceil)$.\footnote{Recall from Definition \ref{reduced-k-log-can} that equivalently this means
$I_0^\prime (D) = \cdots = I_{k-1}^\prime (D) = \shO_X$. By convention the condition is vacuous when $k =0$.}
Then we have:

\begin{enumerate}
\item If $k \le n$, and $L (pZ - \lceil D \rceil)$ is ample for all $2 \le p \le k+1$, then 
$$H^i \big(X, \omega_X \otimes L ((k+1)Z)  \otimes I_k (D) \big) = 0$$
for all $i \ge 2$. Moreover, 
$$H^1 \big(X, \omega_X \otimes L ((k+1)Z) \otimes I_k (D) \big) =  0$$
holds if $H^j \big(X, \Omega_X^{n-j} \otimes L ((k - j +2)Z - \lceil D \rceil )\big) = 0$ for all $1 \le j \le k$.
\medskip

\item If $k \ge n+1$, then $Z$ must be smooth by Corollary \ref{cor_singular_nontrivial}, 
and so $I_k (D) = \shO_X (Z - \lceil D \rceil)$ by Corollary~\ref{smooth_case}.
 In this case, if $L$ is a line bundle such that $L ((k+1) Z -  \lceil D \rceil)$ is ample, then 
$$H^i \big(X, \omega_X \otimes L ((k+1)Z) \otimes I_k (D) \big) = 0 \,\,\,\,\,\, {\rm for ~all}\,\,\,\, i >0.$$

\medskip

\item If $U = X \smallsetminus Z$ is affine (e.g. if $D$ or $Z$ are ample), then (1) and (2) also hold with $L = M (- Z)$, assuming 
that $M(pZ - \lceil D \rceil)$ is ample for $1 \le p \le k$.\footnote{When $k \ge 1$, the condition of $U$ being affine is in fact
implied by the positivity condition, since $D + Z - \lceil D\rceil$ is then an ample divisor with support $Z$.}
\end{enumerate}
\end{theorem}
\begin{proof}
We use the notation in \S\ref{global_setting} and Remark \ref{old_ideal}. In particular, we consider the filtered left $\Dmod_X$-module 
$$\Mmod_1 = M \otimes_{\shO_X} \shO_X (*Z),$$
which we know is a direct summand in a filtered $\Dmod$-module underlying a  mixed Hodge module on $X$.
Its filtration satisfies
$$F_k \Mmod_1 \simeq M(-Z) \otimes \shO_X \big((k+2)Z - \lceil D \rceil\big) \otimes I_k^\prime (D).$$
Note also that since $L + Z - D$ is ample, there exists an ample line bundle $A$ on $X$ such that $L \simeq M(-Z) \otimes A$.

Let's prove (1), i.e. consider the case $k \le n$. 
The statement is equivalent to the vanishing of the cohomology groups 
$$H^i \big(X, \omega_X \otimes L ((k+2)Z - \lceil D \rceil)  \otimes I_k^\prime (D) \big) = 0$$
Since $I_{k-1}^\prime (D)  = \shO_X$, we have a short exact sequence
$$0 \longrightarrow \omega_X \otimes L ((k+1)Z - \lceil D \rceil) \longrightarrow 
\omega_X \otimes L ((k+2)Z - \lceil D \rceil)  \otimes I_k^\prime (D)
\longrightarrow$$ 
$$\longrightarrow \omega_X \otimes A \otimes \gr_k^F \Mmod_1 \longrightarrow 0.$$
By taking the corresponding long exact sequence in cohomology and using Kodaira vanishing, we see that the vanishing we are aiming for is equivalent to the same 
statement for 
$$H^i \big (X,  \omega_X \otimes A \otimes  \gr_k^F \Mmod_1 \big).$$ 

We now consider the complex 
$$C^{\bullet}  : =  \big( \gr_{-n+k}^F \DR(\Mmod_1) \otimes A \big) [-k].$$
Given the hypothesis on the ideals $I_p^\prime (D)$, this can be identified with a complex of the form
$$\big[\Omega_X^{n-k}  \otimes L (2Z- \lceil D \rceil) \longrightarrow  \Omega_X^{n-k+1} \otimes L \otimes \shO_Z(3Z- \lceil D \rceil)  \longrightarrow \cdots $$
$$\cdots \longrightarrow 
\Omega_X^{n-1} \otimes L \otimes \shO_Z\big((k+1)Z - \lceil D \rceil\big)  \longrightarrow \omega_X \otimes A \otimes 
\gr_k^F \Mmod_1  \big]$$
placed in degrees $0$ up to $k$. Saito's Vanishing theorem \cite[\S2.g]{Saito-MHM} gives
\begin{equation}\label{van2}
\mathbf{H}^j ( X, C^{\bullet}) = 0 \quad\text{for all}\quad j \ge k+1.
\end{equation}
We use the spectral sequence 
$$E_1^{p,q} = H^q  (X, C^p) \implies  \mathbf{H}^{p+q} (X, C^{\bullet}).$$
The vanishing statements we are interested in are for the terms $E^{k, i}_1$ with $i \ge 1$. 
We will in fact show that 
\begin{equation}\label{E-terms}
E^{k, i}_r = E^{k, i}_{r +1}, \,\,\,\,\,{\rm for ~all}\,\,\,\,\,\,r \ge 1.
\end{equation}
This implies that 
$$E^{k, i}_1 = E^{k, i}_{\infty} = 0,$$
where the vanishing follows from ($\ref{van2}$) since $i \ge 1$, and this gives our conclusion.

We are thus left with proving ($\ref{E-terms}$). Now on one hand we always have 
$E^{k+r, i -r +1}_r = 0$ because $C^{k+r}=0$. On the other hand, we will show that under our
hypothesis we have $E^{k-r, i +r -1}_1 = 0$, from which we infer that $E^{k-r, i +r -1}_r = 0$ as well, 
allowing us to conclude. To this end, note first that if $r > k$ this vanishing is clear, since the complex $C^\bullet$ starts 
in degree $0$.  If $k = r$, we have 
$$E^{0, i +k -1}_1 = H^{i + k -1} \big( X, \Omega_X^{n-k} \otimes L (2Z - \lceil D \rceil) \big).$$
If $i \ge 2$ this is $0$ by Nakano vanishing, while if $i =1$ it is $0$ because of our hypothesis.
Finally, if $k \ge r + 1$, we have
$$E^{k - r, i +r -1}_1 = H^{i + r -1} \big( X, \Omega_X^{n-r} \otimes L \otimes \shO_Z ((k-r+2)Z - \lceil D \rceil) \big),$$
which sits in an exact sequence
$$H^{i + r - 1} \big( X, \Omega_X^{n- r} \otimes L ((k- r +2) Z - \lceil D \rceil) \big) \longrightarrow E^{k-r, i- r +1}_1 \longrightarrow $$
$$\longrightarrow H^{i  +r} \big( X, \Omega_X^{n- r} \otimes L ((k- r +1 ) Z - \lceil D \rceil) \big).$$
We again have two cases:
\begin{enumerate}
\item If $i \ge 2$, we deduce that $E^{k-r, i +r -1}_1 = 0$ by Nakano vanishing. 
\item If $i = 1$, using Nakano vanishing we obtain a surjective morphism 
$$H^r \big( X, \Omega_X^{n-r} \otimes L ((k-r + 2)Z - \lceil D \rceil) \big) \longrightarrow E^{k-r, i + r -1}_1,$$
and if the extra hypothesis on the term on the left holds, then we draw the same conclusion as in (1).
\end{enumerate}

The same argument proves (3), once we replace Saito Vanishing ($\ref{van2}$) by the vanishing 
$$\mathbf{H}^i \bigl( X, \gr_k^F \DR(\Mmod_1) \bigr) = 0$$
for all $i > 0$ and all $k$,  which in turn is implied by the same statement for the $\Dmod_X$-module $\Mmod$ underlying a Hodge $\Dmod$-module, in which $\Mmod_1$ is a direct summand. Furthermore, this is implied by the vanishing of the perverse sheaf cohomology
$$H^i \big(X, \DR(\Mmod) \bigr) = 0 \,\,\,\,\,\, {\rm for~all~} \,\,\,\,i > 0.$$
Indeed, by the strictness property for direct images (see e.g. \cite[Example~4.2]{MP1}), for $(\Mmod, F)$ we have the decomposition
$$H^i  \big(X, \DR(\Mmod) \bigr)\simeq \bigoplus_{q \in \ZZ}  \mathbf{H}^{i}  
\big(X, \gr_{-q}^F \DR ( \Mmod) \big).$$
Recall now from \S\ref{global_setting} that $\Mmod \simeq j_+ \Nmod$, where $\Nmod$ underlies a Hodge $\Dmod$-module on $U$, and $j \colon U \hookrightarrow X$ is the inclusion. Denoting $P = \DR(\Mmod)$, we then have 
$P \simeq j_* j^* P$, and so it suffices to show that 
$$H^i (U, j^* P ) = 0 \,\,\,\,\,\, {\rm for~all~} \,\,\,\,i > 0.$$
But this is a consequence of Artin vanishing (see e.g. \cite[Corollary~5.2.18]{Dimca}), since $U$ is affine.

Finally, the assertion in (2) follows from Kodaira vanishing, using the long exact sequence in cohomology associated to
the short exact sequence
$$0\to\omega_X\otimes L\big((k+1)Z-\lceil D\rceil\big)\to \omega_X\otimes L\big((k+2)Z-\lceil D\rceil\big)\to
\omega_Z\otimes L\big((k+1)Z-\lceil D\rceil\big)\vert_Z\to 0.$$

\end{proof}

\begin{remark}
We expect the statement of the theorem to hold even without assuming the existence of $M$ (i.e. of an $\ell$-th root of the 
line bundle $\shO_X(\ell D)$). This is known for $k =0$, when the statement follows from Nadel Vanishing, see \cite[Theorem
9.4.8]{Lazarsfeld}. However, at the moment we do not know how to show this for $k \ge 1$.
\end{remark}

\begin{remark}[{\bf Toric varieties}]
As in \cite[Corollary 25.1]{MP1}, when $X$ is a toric variety the Nakano-type vanishing requirement in 
Theorem \ref{vanishing_Hodge_ideals}(1) is automatically satisfied thanks to the Bott-Danilov-Steenbrink vanishing 
theorem. A stronger result in this setting is proved in \cite{Dutta}.
\end{remark}

\begin{remark}[{\bf Projective space, abelian varieties}]
As in \cite[Theorem 25.3 and 28.2]{MP1}, appropriate statements on $\PP^n$ and abelian varieties work without the extra assumptions of reduced log canonicity and Nakano-type vanishing in Theorem \ref{vanishing_Hodge_ideals}. More precisely, keeping the notation at the beginning of the section, we have:

\begin{variant}
Let $D$ be an effective $\QQ$-divisor on $\PP^n$ which is numerically equivalent to a hypersurface of degree $d \ge 1$. 
If $\ell \ge d - n -1$, then 
$$H^i \big(\PP^n, \shO_{\PP^n} (\ell) \otimes \shO_{\PP^n} (kZ) \otimes I_k (D) \big) = 0 \,\,\,\,{\rm for ~all} \,\,\,\, i > 0.$$
\end{variant}

Note that the positivity condition in Theorem \ref{vanishing_Hodge_ideals} is satisfied, since for every effective $\QQ$-divisor 
$D\neq 0$ in $\PP^n$ we have ${\rm deg} \lceil D \rceil <  \deg D + \deg Z$.

\begin{variant}
If $X$ is an abelian variety and $D$ is an ample $\QQ$-divisor on $X$, then 
$$H^i (X, M (kZ) \otimes I_k (D) \otimes \alpha) = 0$$
for all $i > 0$ and $\alpha \in \Pic^0 (X)$.
\end{variant}

Note that on an abelian variety every effective $\QQ$-divisor is nef, and the ampleness of $D$ is equivalent to that of any divisor 
whose support is equal to that of $D$.

The proofs are completely similar to those in \emph{loc. cit.}, replacing $\shO_X (*D)$ in the reduced case by $\Mmod_1$ in the proof above, and noting that since $\Mmod_1$ is a filtered direct summand in $j_+ p_+ \shO_V$ as in \S\ref{global_setting}, the vanishing properties we use continue to hold.
\end{remark}

\section{Restriction, subadditivity, and semicontinuity theorems}\label{section_restriction}

In this part of the paper we provide $\QQ$-divisor analogues of the results in \cite{MP2}. This extends well-known 
statements  in the setting of multiplier ideals; further discussion and references regarding these can be found in \emph{loc. cit.}

\subsection{Restriction theorem}\label{scn:restriction}
We begin with the $\QQ$-divisor version of the Restriction Theorem:

\begin{theorem}\label{restriction}
Let $D$ be an effective ${\mathbf Q}$-divisor, with support $Z$, on the smooth variety $X$, and
let $Y$ be a smooth irreducible divisor on $X$ such that $Y\not\subseteq Z$. If we denote $D_Y=D\vert_Y$, $Z_Y=Z\vert_Y$, and 
$Z'_Y=(Z_Y)_{\rm red}$, then for every $k\geq 0$ we have
\begin{equation}\label{eq1_restriction}
\shO_Y\big(-k(Z_Y-Z'_Y)\big)\cdot I_k(D_Y)\subseteq I_k(D)\cdot\shO_Y.
\end{equation}
In particular, if $Z_Y$ is reduced, then for every $k\geq 0$ we have
\begin{equation}\label{eq2_restriction}
I_k (D_{Y}) \subseteq I_k(D) \cdot \shO_Y.
\end{equation}
Moreover, if $Y$ is sufficiently general (e.g. a general member of a basepoint-free linear system), then 
we have equality in (\ref{eq2_restriction}).
\end{theorem}

\begin{remark}
Note that when $D$ is a reduced divisor we have $D_Y = Z_Y$, and $D_Y - Z'_Y$ is an integral divisor with support in $Z'_Y$. Therefore Lemma \ref{periodicity} gives
$$I_k (D_Y) = \shO_X\big(-(D_Y-Z'_Y)\big)\cdot I_k(Z'_Y),$$
hence the statement in the theorem coincides with that of \cite[Theorem~A]{MP2}.
\end{remark}

\begin{proof}[Proof of Theorem~\ref{restriction}]
The argument follows the proof of \cite[Theorem~A]{MP2}, with a simplification observed in \cite{Saito-MLCT},
hence we only give the outline of the proof.
Since the statement is local, we may assume that $D=\alpha\cdot {\rm div}(h)$
for some nonzero $h\in\shO_X(X)$.
Consider the following commutative diagram with Cartesian squares:
$$
\begin{tikzcd}
V_Y \rar{i''} \dar{p'} & V\dar{p} \\
U_Y \rar{i'} \dar{j'} & U \dar{j} \\
Y \rar{i} & X,
\end{tikzcd}
$$
where $p$ and $j$ are as in diagram (\ref{cart_diag1}), while $i$ is the inclusion of $Y$ in $X$.
Note that if $n=\dim(X)$, we have a canonical base-change isomorphism
$$ i^! (j\circ p)_+{\mathbf Q}_V^H[n] \simeq (j' \circ p')_+  {i^{''}}^!{\mathbf Q}_V^H[n]$$
proved in \cite[4.4.3]{Saito-MHM}. We also have a canonical isomorphism 
$${i^{''}}^! \QQ_V^H [n] = (\QQ_{V_Y}^H [n-1]) (-1)[-1]$$ 
(see for instance 
\cite[\S3.5]{Saito-MHP}).
Here we use the Tate twist notation, which for a mixed Hodge module $M = (\Mmod, F_\bullet \Mmod, K)$ is given by 
$$M (k)= \big(\Mmod, F_{\bullet - k} \Mmod, K \otimes_{\QQ} \QQ(k)\big).$$
We obtain, in particular, an isomorphism of filtered right $\Dmod_X$-modules
$$\big({\mathcal H}^1i^!\Mmod_r(h^{-\alpha}),F_{\bullet}\big)\simeq\big(\Mmod_r(h\vert_Y^{-\alpha}),F_{\bullet+1}\big).$$

Recall now that if $(V_{\alpha}\Mmod)_{\alpha\in\QQ}$ is the $V$-filtration on $\Mmod=\Mmod_r(h^{-\alpha})$ corresponding to the
smooth hypersurface $Y\subseteq X$, then there is a canonical morphism
$$\sigma\colon {\rm gr}^V_0\Mmod\to {\rm gr}^V_{-1}\Mmod\otimes_{\shO_X}\shO_X(Y)$$
such that 
$${\mathcal H}^1i^!\Mmod\simeq {\rm coker}(\sigma),$$
with the Hodge filtration on the right-hand side induced by the Hodge filtration on $\Mmod$.
We refer to \cite[\S2]{MP2} for details. 

One defines a morphism 
$$\eta\colon F_k{\rm gr}^V_{-1}\Mmod=\frac{F_kV_{-1}\Mmod}{F_kV_{<-1}\Mmod}\longrightarrow F_k\Mmod\otimes_{\shO_X}\shO_Y$$
that maps the class of $u\in F_kV_{-1}\Mmod=F_k\Mmod\cap V_{-1}\Mmod$ to the class of $u$ in $F_k\Mmod\otimes_{\shO_X}\shO_Y$.
After tensoring $\eta$ with $\shO_X(Y)$, the resulting morphism vanishes on the image of the restriction of $\sigma$ to 
$F_k{\rm gr}^V_0\Mmod$, hence we obtain an induced morphism 
\begin{equation}\label{eq2_restriction_thm}
F_{k+1}\Mmod_r(h\vert_Y^{-\alpha})\simeq F_k{\mathcal H}^1i^!\Mmod\simeq F_k{\rm coker}(\sigma)\to F_k\Mmod\otimes_{\shO_X}\shO_Y(Y).
\end{equation}
Applying this with $k$ replaced by $k-n$, it follows from the definition of Hodge ideals and the formula for the equivalence between
left and right $\Dmod$-modules that we have a morphism
$$I_k(D_Y)\otimes_{\shO_Y}\omega_Y\big(kZ'_Y+{\rm div}(h\vert_Y)\big)\to
I_k(D)\otimes_{\shO_X}\omega_X\big(kZ+{\rm div}(h)\big)\otimes_{\shO_X}\shO_Y(Y).$$
By tensoring this with $\omega_Y^{-1}\big(-kZ_Y-{\rm div}(h\vert_Y)\big)$
and composing with the canonical map $I_k(D)\otimes_{\shO_X}\shO_Y\to I_k(D)\cdot\shO_Y$, we obtain a canonical
morphism
$$
%\begin{equation}\label{eq_restriction_thm}
\varphi\colon \shO_Y\big(-k(Z_Y-Z'_Y)\big)\otimes I_k(D_Y)\to I_k(D)\cdot\shO_Y.
%\end{equation}
$$
Note that all constructions are compatible with restrictions to open subsets and when restricting to $Z=X\smallsetminus U$,
the above morphism can be identified with the identity map on $\shO_Y$. Therefore the morphism $\varphi$ is compatible with the
two inclusions in $\shO_Y$, and we deduce the inclusion in (\ref{eq1_restriction}).

Suppose now that $Y$ is general, so that $Z_Y=Z'_Y$ and $Y$ is non-characteristic with respect to $\Mmod$.
For example, this condition holds if $Y$ is transversal to the strata in a Whitney stratification of $Z$ (see 
\cite[\S2]{Dimca_et_al}); in particular, it holds if $Y$ is a general member of a basepoint-free linear system.
We may assume that $Y$ is defined by a global equation $t\in\shO_X(X)$.
In this case, it follows from \cite[Lemme~3.5.6]{Saito-MHP} that ${\rm gr}^V_0\Mmod=0$ and ${\rm gr}_{-1}^V\Mmod=
\Mmod\otimes_{\shO_X}\shO_Y$. It is now straightforward to check that the morphism (\ref{eq2_restriction_thm})
is an isomorphism, hence $\varphi$ is an isomorphism, and we thus have equality in (\ref{eq2_restriction}).
\end{proof}

We deduce the following analogue of inversion of adjunction:

\begin{corollary}\label{cor_restriction}
With the notation of Theorem \ref{restriction}, if $Z_Y$ is reduced and  $I_k (D_Y)_x = \shO_{Y,x}$
for some $x \in Y$,  then $I_k (D)_x = \shO_{X,x}$. 
\end{corollary}

\begin{remark}\label{restriction_higher_codimension}
If $D$ is an effective ${\mathbf Q}$-divisor, with support $Z$, on the smooth variety $X$, and
$Y$ is a smooth subvariety of $X$ such that $Y\not\subseteq Z$ and $Z\vert_Y$ is reduced, then
for every $k\geq 0$ we have
$$I_k(D\vert_Y)\subseteq I_k(D)\cdot\shO_Y.$$
This follows by writing $Y$ locally as a transverse intersection of $r$ smooth divisors on $X$
and applying repeatedly the inclusion (\ref{eq2_restriction}). 
\end{remark}

\begin{remark}
With the notation in Theorem~\ref{restriction}, let $Y_1,\ldots,Y_r$ be general elements in a basepoint-free linear 
system on $X$, where $r\leq n=\dim(X)$.  If $W=Y_1\cap\cdots\cap Y_r$, then for every $k\geq 0$ we have
$$I_k(D\vert_W)=I_k(D)\cdot\shO_W.$$
Indeed, if $W_i=Y_1\cap\cdots\cap Y_i$, and if $(S_{\beta})_{\beta}$ are the strata of a Whitney stratification of $Z$, then it follows
by induction on $i$ that we have a Whitney stratification of $Z\vert_{W_i}$ with strata $(S_{\beta}\cap W_i)_{\beta}$. 
Moreover, $Y_{i+1}$ is transversal to each such stratum. We may thus apply the theorem to each divisor $D\vert_{W_i}$ and smooth hypersurface $Y_{i+1}\cap W_i\subseteq W_i$, to conclude that
$$I_k(D\vert_W)=I_k(D)\cdot\shO_W.$$
\end{remark}

\subsection{Semicontinuity theorem}\label{scn:semicontinuity}

The same argument as in \cite[\S5]{MP2}, based on the Restriction Theorem (in this case Theorem \ref{restriction} above), gives the following semicontinuity statement.
The set-up is as follows:  let $f\colon X\to T$ be a smooth morphism of relative dimension $n$
between arbitrary varieties $X$ and $T$, and $s\colon T\to X$ a morphism such that $f\circ s={\rm id}_T$.
Let $D$ be an effective $\QQ$-Cartier $\QQ$-divisor on $X$, relative over $T$ (that is, we can write $D$ locally
as $\alpha H$, for an effective divisor $H$ and a positive rational number $\alpha$, with $H$ flat over $T$). We assume that we have an effective
divisor $Z$ on $X$, relative over $T$, with ${\rm Supp}(Z)={\rm Supp}(D)$, and such that for every $t\in T$, the
restriction $Z_t$ to the fiber $X_t=f^{-1}(t)$ is reduced. 
For every $x\in X$, we denote by $\mathfrak{m}_x$ the ideal defining $x$ in $X_{f(x)}$.

\begin{theorem}\label{semicontinuity}
With the above notation,  for every $q\geq 1$, the set
$$V_q:=\big\{t\in T\mid I_k(D_t)\not\subseteq \mathfrak{m}_{s(t)}^q\big\},$$
is open in $T$. 
\end{theorem}

\subsection{Subadditivity theorem}\label{scn:subadditivity}

The calculation for $I_2$ in Example \ref{parameter_dependence} shows that the inclusion 
$$I_k (D_1 + D_2) \subseteq I_k (D_1)$$ 
cannot hold for arbitrary $\QQ$-divisors $D_1$ and $D_2$. However, with an appropriate assumption on the support, we have the 
following stronger subadditivity statement: 

\begin{theorem}\label{subadditivity}
If $D_1$ and $D_2$ are effective $\QQ$-divisors on the smooth variety $X$, whose supports $Z_1$ and $Z_2$
satisfy the property that $Z_1+Z_2$ is reduced, then for every $k\geq 0$ 
we have
$$I_k(D_1+D_2)\subseteq\sum_{i+j=k}I_i(D_1)\cdot I_j(D_2)\cdot \shO_X(-jZ_1-iZ_2)\subseteq 
I_k(D_1)\cdot I_k(D_2).$$
\end{theorem}

Note first that, for every $i$ and $j$, the inclusion 
$$F_i \Mmod (h^{- \alpha}) \subseteq F_{i+j} \Mmod (h^{- \alpha})$$
implies the inclusion 
\begin{equation}\label{easy_inclusion}
\shO_X (- jZ) \cdot I_i (D) \subseteq I_{i +j} (D).
\end{equation}
This gives the second inclusion in the statement above.
To prove the first inclusion, as in the proof of \cite[Theorem~B]{MP2} it is enough to show the following:\footnote{Indeed, the Restriction Theorem applies in the form given in 
Remark~\ref{restriction_higher_codimension} for the diagonal embedding $X\hookrightarrow X\times X$, since we are assuming that $Z_1 + Z_2$ is reduced.}

\begin{proposition}\label{product_case}
Let $X_1$ and $X_2$ be smooth varieties and let $D_i$ be effective $\QQ$-divisors on $X_i$, with support 
$Z_i$, for $i=1,2$. If $B_i=p_i^* D_i$,
where $p_i\colon X_1\times X_2\to X_i$ are the canonical projections, then for every $k\geq 0$ we have
$$I_k(B_1+B_2)=\sum_{i+j=k}\big(I_i(D_1)\shO_{X_1}(-jZ_1)\cdot\shO_{X_1\times X_2}\big)\cdot \big(I_j(D_2)\shO_{X_2}(-iZ_2)\cdot\shO_{X_1\times X_2}\big).$$
\end{proposition}

\begin{proof}
By Remark \ref{eq_renormalization_0}, we can assume that there exist regular functions $h_1$ on $X_1$ and 
$h_2$ on $X_2$, together with $\alpha \in \QQ_{> 0}$, such that $I_i (D_1)$ and $I_j (D_2)$ are defined by 
$\Mmod_r (h_1^{-\alpha})$ and $\Mmod_r (h_2^{-\alpha})$, respectively. The statement follows precisely as in \cite[Proposition~4.1]{MP2}, 
as long as we show that there is a canonical isomorphism of filtered $\Dmod$-modules
$$\big( \Mmod_r ( (p_1^*h_1 \cdot p_2^* h_2)^{-\alpha}), F\big) \simeq \big( \Mmod_r ( h_1^{-\alpha}) \boxtimes \Mmod_r ( h_2^{-\alpha}) , F\big),$$
where the filtration on the right hand side is the exterior product of the filtrations on the two factors.
But this is a consequence of the canonical isomorphism of mixed Hodge modules 
$$j_* p_* \QQ_{V_1 \times V_2}^H [n_1 + n_2] \simeq {j_1}_* {p_1}_* \QQ_{V_1}^H [n_1]  \boxtimes
{j_2}_* {p_2}_* \QQ_{V_2}^H  [n_2],$$
with the obvious notation as in ($\ref{cart_diag1}$) for $i =1,2$, together with Lemma \ref{decomp_filtration}. 
\end{proof}


\section*{References}
\begin{biblist}
\bib{Bjork}{book}{
       author={Bj{\"o}rk, Jan-Erik},
       title={Analytic $\Dmod$-modules and applications},  
       series={Mathematics and its Applications},  
       publisher={Kluwer Academic Publishers},
       date={1993},
}  
\bib{Dimca}{book}{
       author={Dimca, Alexandru},
       title={Sheaves in topology},  
       series={Universitext},
       publisher={Springer-Verlag, Berlin},
   date={2004},
}      

\bib{Dimca_et_al}{article}{
 author={Dimca, Alexandru},
   author={Maisonobe, Philippe},
   author={Saito, Morihiko},
   author={Torrelli, Tristan},
   title={Multiplier ideals, $V$-filtrations and transversal sections},
   journal={Math. Ann.},
   volume={336},
   date={2006},
   number={4},
   pages={901--924},
   }

\bib{Dutta}{article}{
      author={Dutta, Yajnaseni},
	title={Vanishing for Hodge ideals on toric varieties},
	journal={in preparation}, 
	date={2018}, 
}
\bib{EV}{book}{
       author={Esnault, H\'el\`ene},
       author={Viehweg, Eckart},
       title={Lectures on vanishing theorems},  
       series={DMV Seminar},  
       volume={20},
       publisher={Birkh\"auser},
       date={1992},
}  

     
\bib{Fulton}{book}{
   author={Fulton, William},
   title={Introduction to toric varieties},
   series={Annals of Mathematics Studies},
   volume={131},
   note={The William H. Roever Lectures in Geometry},
   publisher={Princeton University Press, Princeton, NJ},
   date={1993},
}


\bib{GKKP}{article}{
   author={Greb, Daniel},
   author={Kebekus, Stefan},
   author={Kov{\'a}cs, S{\'a}ndor J.},
   author={Peternell, Thomas},
   title={Differential forms on log canonical spaces},
   journal={Publ. Math. Inst. Hautes \'Etudes Sci.},
   number={114},
   date={2011},
   pages={87--169},
}


\bib{HTT}{book}{
   author={Hotta, R.},
   author={Takeuchi, K.},
   author={Tanisaki, T.},
   title={D-modules, perverse sheaves, and representation theory},
   publisher={Birkh\"auser, Boston},
   date={2008},
}


%\bib{KMM}{article}{
 %     author={Kawamata, Yujiro},
%      author={Matsuda, Katsumi},
%      author={Matsuki, Kenji},
%      title={Introduction to the minimal model program},
%      journal={Advanced Studies in Pure Math.}, 
%      number={10},
%      date={1987}, 
%      pages={283--360},
%}
\bib{Lazarsfeld}{book}{
       author={Lazarsfeld, Robert},
       title={Positivity in algebraic geometry II},  
       series={Ergebnisse der Mathematik und ihrer Grenzgebiete},  
       volume={49},
       publisher={Springer-Verlag, Berlin},
       date={2004},
}      
\bib{MP1}{article}{
      author={Musta\c t\u a, Mircea},
      author={Popa, Mihnea},
	title={Hodge ideals},
	journal={preprint arXiv:1605.08088, to appear in Memoirs of the AMS}, 
	date={2016}, 
}
\bib{MP2}{article}{
      author={Musta\c t\u a, Mircea},
      author={Popa, M.},
      title={Restriction, subadditivity, and semicontinuity theorems for Hodge ideals},
      journal={Int. Math. Res. Not.}, 
      date={2018}, 
      number={11},
      pages={3587--3605},
}

\bib{MP3}{article}{
      author={Musta\c t\u a, Mircea},
      author={Popa, Mihnea},
      title={Hodge ideals for $\QQ$-divisors, $V$-filtration, and minimal exponent},
      journal={preprint arXiv:1807.01935}, 
      date={2018}, 
}

\bib{Popa}{article}{
      author={Popa, Mihnea},
      title={$\Dmod$-modules in birational geometry},
      journal={to appear in the Proceedings of the ICM, Rio de Janeiro}, 
      date={2018}, 
}


\bib{PS}{article}{
      author={Popa, Mihnea}, 
      author={Schnell, Christian}, 
      title={Generic vanishing theory via mixed Hodge modules}, 
      journal={Forum Math. Sigma},
      volume={1},
      date={2013}, 
      pages={60 pp},
}

\bib{Saito-MHP}{article}{
   author={Saito, Morihiko},
   title={Modules de Hodge polarisables},
   journal={Publ. Res. Inst. Math. Sci.},
   volume={24},
   date={1988},
   number={6},
   pages={849--995},
}

\bib{Saito-MHM}{article}{
   author={Sa{}ito, Morihiko},
   title={Mixed Hodge modules},
   journal={Publ. Res. Inst. Math. Sci.},
   volume={26},
   date={1990},
   number={2},
   pages={221--333},
}
\bib{Saito-LOG}{article}{
   author={Saito, Morihiko},
   title={Direct image of logarithmic complexes and infinitesimal invariants of cycles},
   conference={
      title={Algebraic cycles and motives. Vol. 2},
   },
   book={
      series={London Math. Soc. Lecture Note Ser.},
      volume={344},
      publisher={Cambridge Univ. Press, Cambridge},
   },
   date={2007},
   pages={304--318},
   }
   
      \bib{Saito-HF}{article}{
   author={Saito, M.},
   title={On the Hodge filtration of Hodge modules},
   journal={Mosc. Math. J.},
   volume={9},
   date={2009},
   number={1},
   pages={161--191},
}

   
\bib{Saito-MLCT}{article}{
      author={Saito, Morihiko},
      title={Hodge ideals and microlocal $V$-filtration},
      journal={preprint arXiv:1612.08667}, 
      date={2016}, 
}

\bib{Saito-YPG}{article}{
   author={Saito, M.},
   title={A young person's guide to mixed Hodge modules},
   conference={
      title={Hodge theory and $L^2$-analysis},
   },
   book={
      series={Adv. Lect. Math. (ALM)},
      volume={39},
      publisher={Int. Press, Somerville, MA},
   },
   date={2017},
   pages={517--553},
}
	

\bib{Viehweg}{article}{
      author={Viehweg, Eckart},
      title={Vanishing theorems},
      journal={J. Reine Angew. Math.},
      number={335},
      date={1982},
      pages={1--8},
}
\bib{Zhang}{article}{
      author={Zhang, Mingyi},
      title={Hodge filtration for $\QQ$-divisors with quasi-homogeneous isolated singularities},
      journal={in preparation}, 
      date={2018}, 
}
\end{biblist}
\end{document}